\documentclass[11pt,a4paper,reqno]{amsart}
\usepackage{amsmath,amsthm,amsfonts,amssymb, mathabx, mathrsfs}
\usepackage[abbrev,non-sorted-cites]{amsrefs}
\usepackage{latexsym}
\usepackage{color}
\usepackage{tikz, quiver}
\usepackage{subcaption, sidecap}
\usepackage{graphicx}
\usepackage{verbatim}
\usepackage{appendix}
\usepackage{slashed}
\usepackage{comment}

\usepackage{hyperref}
\hypersetup{
    colorlinks=true,
    linkcolor=black,
    citecolor=black,
}

\DeclareMathOperator{\im}{Im}
\DeclareMathOperator{\re}{Re}

\DeclareMathOperator{\U}{U}
\DeclareMathOperator{\SU}{SU}
\DeclareMathOperator{\SO}{SO}
\DeclareMathOperator{\Stab}{Stab}

\oddsidemargin=.in
\evensidemargin=.in

\newtheorem{theorem}{Theorem}[section]
\theoremstyle{plain}

\newtheorem{assumption}{Assumption}
\newtheorem{ansatz}{Ansatz}

\newtheorem{corollary}[theorem]{Corollary}

\newtheorem{example}{Example}

\newtheorem{lemma}[theorem]{Lemma}

\newtheorem{proposition}[theorem]{Proposition}
\newtheorem{remark}[theorem]{Remark}

\title{Lagrangian Translating Solitons and Special Lagrangians in $\mathbb{C}^{m}$ with Symmetries}

\author{Wei-Bo Su}
\address{National Center for Theoretical Sciences, Mathematics Division, Taipei 106319, Taiwan}
\email{weibo.su@ncts.tw}

\author{Albert Wood}
\address{Department of Mathematics, The Chinese University of Hong Kong, Ma Liu Shui, Hong Kong}
\email{albertwood@math.cuhk.edu.hk}

\date{}

\begin{document}

\begin{abstract}
We construct novel families of exact immersed and embedded Lagrangian translating solitons and special Lagrangian submanifolds in $\mathbb{C}^m$ that are invariant under the action of various admissible compact subgroups $G \leq 
\SU(m-1)$ with cohomogeneity-two. These examples are obtained via an Ansatz generalising a construction of Castro--Lerma in $\mathbb{C}^2$. We give explicit examples of admissible group actions, including a full classification for $G$ simple. 

We also describe novel Lagrangian translators symmetric with respect to non-compact subgroups of the affine special unitary group $\SU(m)\ltimes \mathbb{C}^m$, including cohomogeneity-one examples.
\end{abstract}

\maketitle

\section{Introduction}

Special Lagrangian submanifolds of Calabi-Yau manifolds play a central role in the fields of symplectic geometry and mirror symmetry, as the expected mirrors of holomorphic bundles via the SYZ conjecture. As volume-minimising calibrated submanifolds, they are also of interest in differential geometry and geometric analysis. Their classification and construction are therefore of fundamental interest to a wide variety of fields. Since Smoczyk's discovery that the mean curvature flow preserves the class of Lagrangian submanifolds, conjectures have been made about long-time existence and convergence of the flow to special Lagrangian representatives, and \textit{Lagrangian mean curvature flow} has become a rich field of study.

Of particular importance is singularity analysis for the flow, particularly in the zero-Maslov\footnote{A smooth Lagrangian is \textit{zero-Maslov} if the cohomology class of the mean-curvature one form is trivial.} case. It is conjectured that singularities of zero-Maslov Lagrangian mean curvature flow in Calabi-Yau manifolds must be modelled on translating solitons or special Lagrangians in $\mathbb{C}^m$. For example, a criterion for a singularity model to be a translating soliton is given by \cite{LSSz24}, and the possibility of shrinking solitons as singularity models was excluded by work of Wang \cite{Wang2001} and Neves \cite{Neves2007}. Investigation and classification of special Lagrangians and Lagrangian translating solitons in $\mathbb{C}^m$ is therefore of great importance to the study of Lagrangian mean curvature flow, and its related conjectures.

In this paper, we construct new Lagrangian translating solitons and special Lagrangians in $\mathbb{C}^{m}$ invariant under symmetry groups $G \leq \SU(m-1)$ with cohomogeneity at most two.\footnote{See Section \ref{sec-liegroupactions} for the definition of cohomogeneity-$k$.} Soliton solutions to LMCF with cohomogeneity-one symmetry were studied extensively in Madnick-Wood \cite{Wood2022}; in particular for a \textit{linear} group action $G \leq \SU(m)$ the cohomogeneity-one special Lagrangians, Lagrangian shrinkers and Lagrangian expanders were classified, and it was shown that there are no cohomogeneity-one Lagrangian translators. The main theme in the cohomogeneity-one theory is that after symmetry reduction, the defining equations of the solitons reduce to second-order ODEs satisfied by curves in a profile plane $\mathbb{C}\simeq\mathbb{R}^{2}$. In the cohomogeneity-two case, we show that after symmetry reduction, the equations become PDEs satisfied by \emph{Lagrangian surfaces} in $\mathbb{C}^{2}$ (equations \eqref{eq-cohom2translator}, \eqref{eq-cohom2SL}). The PDEs can be viewed as weighted versions of the special Lagrangian and Lagrangian translator equations in $\mathbb{C}^{2}$, which are still difficult to solve. However, we take advantage of the complex 2-dimensional nature of the reduced problems and construct explicit solutions to the weighted equations using a \emph{generalised Castro--Lerma Ansatz}, inspired by the work of Castro--Lerma \cites{Castro2010, CL14}. 

The Ansatz is as follows:

\begin{ansatz}\label{ans-1}
    Let $m \geq 2$, $I_{1}, I_{2}\subseteq\mathbb{R}$ be open intervals and $\gamma:I_{1}\to\mathbb{C}$, $\xi:I_{2}\to\mathbb{C}$ be smooth curves, such that $\gamma$ and $\xi$ satisfy
    \begin{equation}\label{eq-curveeqsintro}
        \vec{\kappa}_{\gamma} - (m-2)\frac{\gamma^{\perp}}{|\gamma|^{2}} = a\gamma^{\perp},\quad \vec{\kappa}_{\xi} - (m-2)\frac{\xi^{\perp}}{|\xi|^{2}} = -a\xi^{\perp},
    \end{equation}
    for some $a\in\mathbb{R}$. Given a subgroup $G \leq \SU(m-1)$ and a linear embedding $\Psi:\mathbb{C}^2 \to \mathbb{C}^m$, define the maps
    \begin{align}
       \check F:I_1 \times I_2 &\to \mathbb{C}^2, \quad &&\check F(x,y) \, :=\, \left( \gamma(x)\xi(y), \, \int_0^y \overline {\xi(s)} \xi'(s)\, ds \, - \, \int_0^x \overline {\gamma(s)} \gamma'(s)\,ds \right) \label{eq-checkF}\\
       F:G \times I_1 \times I_2 &\to \mathbb{C}^m, &&F(g, x, y) := g \cdot \Psi\circ \check F(x,y). \label{eq-F}
    \end{align}
\end{ansatz}
\noindent We note that if $m=2$, then Ansatz \ref{ans-1} reduces to the construction of Castro--Lerma.

For suitable $G$ and $\Psi$, this Ansatz produces $G$-invariant Lagrangian translators and special Lagrangians. To state this precisely, we need a couple of definitions. For a subgroup $G \leq \SU(m-1)$, we define the \textit{standard moment map}\footnote{This is a moment map for the $G$-action on $\mathbb{C}^{m-1}$ in the sense of Hamiltonian actions, after identifying $\mathfrak{g}$ with $\mathfrak{g}^*$ via an Ad-invariant inner product.} for $G$ to be
\begin{equation} \label{eq-standardmomentmap}
    \mu:\mathbb{C}^{m-1} \to \mathfrak{g}, \quad \mu(z) = -\frac{1}{2}\pi_\mathfrak{g}(iz\overline{z}^T).
\end{equation} 
We say that a compact, connected subgroup $G \leq \SU(m-1)$ is \textit{admissible} if $\mu^{-1}(0)$ contains a cohomogeneity-one Lagrangian. Examples of admissible groups include $\SO(m-1)$ and the maximal torus $T^{m-2}$ (see Examples \ref{ex-so(m-1)}, \ref{ex-t(m-2)} for details).
\begin{theorem}\label{thm-main}
    For admissible $G$, there exists a linear embedding $\Psi: \mathbb{C}^2 \to \mathbb{C}^m$ such that:
    \begin{itemize}
        \item If $a>0$, then the image of $F$ in Ansatz \ref{ans-1} is an exact Lagrangian translator.
        \item If $a=0$, then the image of $F$ in Ansatz \ref{ans-1} is an exact special Lagrangian.
    \end{itemize}
\end{theorem}

For example, if $G = \SO(m-1)$ then we can choose $\Psi$ to be $\Psi_{\SO(m-1)}(z_1, z_2) := (z_1, 0,\ldots, 0, z_m)$, and if $G = T^{m-2}$ then we can choose $\Psi$ to be $\Psi_{T^{m-2}}(z_1, z_2) := (z_1, z_1, \ldots, z_1, z_2)$. 
Here are some examples of novel and existing special Lagrangians and translators produced by this Ansatz with these choices for $G$. 

\begin{example}[Figure \ref{fig-d1} and   \ref{fig-d2}]\label{ex-d}
    Let $a = 1$, and choose $\gamma(x)$ to be a line through the origin and $\xi(x) = \tilde \sigma(x)$ to be the profile curve of an Anciaux shrinker\footnote{A closed solution $\xi$ to \eqref{eq-curveeqsintro} for $a>0$, originally studied by Anciaux in \cite{Anciaux2006}. See Figure \ref{fig-shrinker} for an example.}. The image of the function $F$ given by Ansatz \ref{ans-1} with $G = \SO(m-1)$ and $\Psi = \Psi_{\SO(m-1)}$ is an embedded $\SO(m-1)$-invariant Lagrangian translator in $\mathbb{C}^m$. It is periodic in the $\emph{Im}(z_m)$-direction, and may be thought of as a Lagrangian translating helicoid.
\end{example}

\begin{example}[Figure \ref{fig-e}]\label{ex-e}
    Let $a = 1$, and choose $\gamma(x) = \tilde\epsilon(x)$ to be the profile curve of an Anciaux expander\footnote{A closed solution $\gamma$ to \eqref{eq-curveeqsintro} for $a>0$. See Figure \ref{fig-expander} for an example.} and $\xi(y)$ to be a line. The image of $F$ with $G = \SO(m-1)$ and $\Psi = \Psi_{\SO(m-1)}$ is an almost-calibrated embedded $\SO(m-1)$-invariant Lagrangian translator in $\mathbb{C}^m$. This recovers the symmetric Joyce--Lee--Tsui translator of \cite{Joyce2010}.
\end{example}

\begin{example}\label{ex-f}
    Let $a =1$, and let $\gamma(x) = \tilde\sigma(x)$ and $\xi(y) = \tilde \epsilon(x)$ be profile curves of an Anciaux shrinker and expander respectively. The image of the function $F$ with $G = \SO(m-1)$ or $T^{m-2}$ and $\Psi = \Psi_{\SO(m-1)}$ or $\Psi_{T^{m-2}}$ is an immersed $\SO(m-1)$ or $T^{m-2}$-invariant Lagrangian translator in $\mathbb{C}^m$. It is unbounded in the $\text{Im}(z_m)$-direction, and not almost-calibrated.
\end{example}

\begin{example}[Figure \ref{fig-b}]\label{ex-b}
    Let $a=0$. The solution $\gamma(x) = \tilde l$ to equation \eqref{eq-curveeqsintro} is the profile curve of a symmetric Lawlor neck.\footnote{The Lawlor neck is the unique exact special Lagrangian in complex Euclidean space asymptotic to two planes, first discovered by Lawlor \cite{Lawlor1989}.} If $\gamma(x)$ is a line and and $\xi(x) = \tilde l$, the image of the function $F$ given by Ansatz \ref{ans-1} with $G = \SO(m-1)$ and $\Psi = \Psi_{\SO(m-1)}$ is an embedded $\SO(m-1)$-invariant special Lagrangian submanifold in $\mathbb{C}^m$. This recovers an example of Joyce \cite[Theorem 7.3]{Joyce01}, also obtainable as a limit of Example 2.
\end{example}

\begin{example}[Figure \ref{fig-c}]\label{ex-c}
    Let $a=0$. If we choose $\gamma(x) = \tilde l_1$ and $\xi(x) = \tilde l_2$ both as Lawlor neck profile curves, the image of the function $F$ given by Ansatz \ref{ans-1} with $G = \SO(m-1)$ or $T^{m-2}$ and $\Psi = \Psi_{\SO(m-1)}$ or $\Psi_{T^{m-2}}$ is an embedded $\SO(m-1)$ or $T^{m-2}$-invariant special Lagrangian submanifold in $\mathbb{C}^m$.
\end{example}

\begin{figure}[p!]
  \centering
  \begin{subfigure}[t]{0.48\textwidth}
    \centering
    \includegraphics[width=\linewidth]{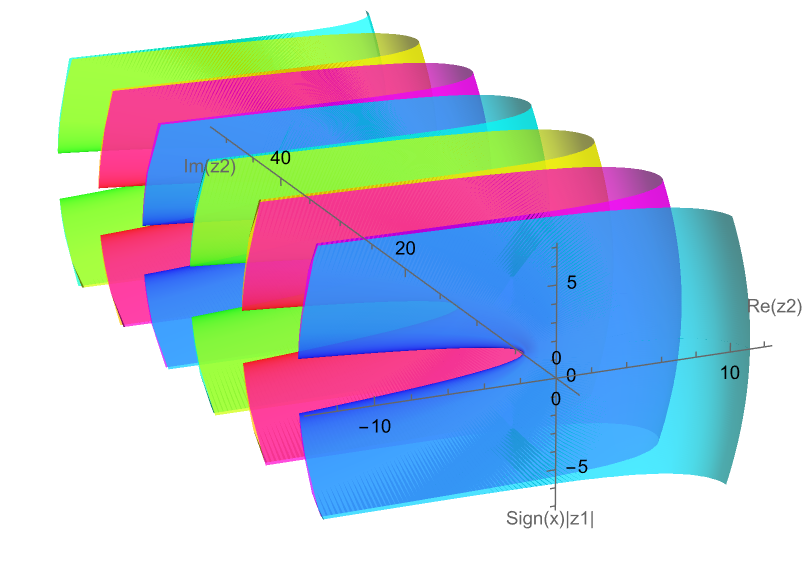}
    \caption{}
    \label{fig-d1}
  \end{subfigure}
  \hfill
  \begin{subfigure}[t]{0.48\textwidth}
    \centering
    \includegraphics[width=\linewidth]{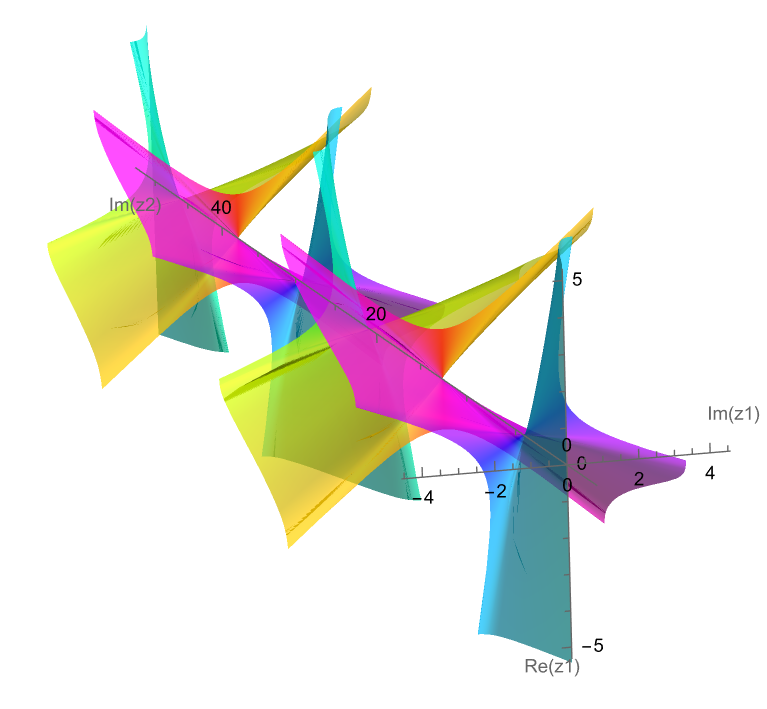}
    \caption{}
    \label{fig-d2}
  \end{subfigure}
  \caption{Two views of the profile surface $\check F_{(d)}$ in $\mathbb{C}^2$ (equation \eqref{eq-Fd}), given by equation \eqref{eq-checkF} with $\gamma(x)$ a line and $\xi(y)$ the shrinker profile in Figure \ref{fig-shrinker}. The surface has topology $\mathbb{R}^2$, and is periodic in the $\text{Im}(z_2)$ direction. The colour of the surfaces corresponds to the argument of the first complex coordinate $z_1$, and is the same in both diagrams.\newline
  By applying Ansatz \ref{ans-1}, this surface in $\mathbb{C}^2$ corresponds to a $G$-invariant Lagrangian translator in $\mathbb{C}^3$, translating in the $\text{Re}(z_3)$ direction (corresponding to $\text{Re}(z_2)$ in the images).}
\end{figure}

\begin{figure}[p!]
  \centering
  \begin{subfigure}[t]{0.48\textwidth}
    \centering
    \includegraphics[width=\linewidth]{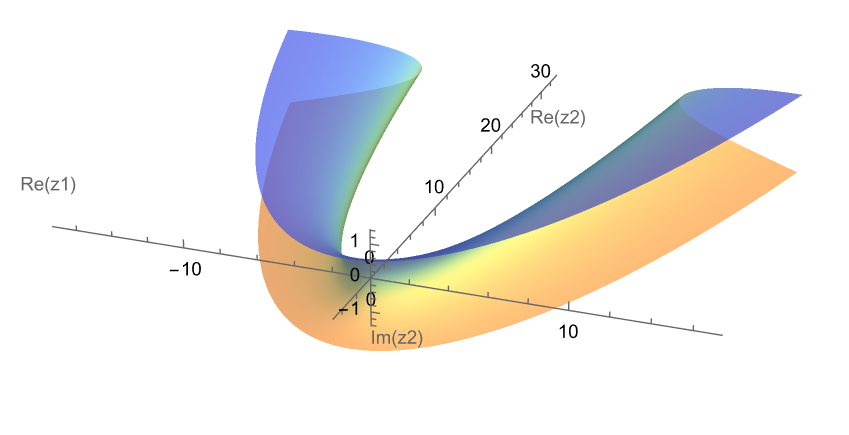}
    \caption{}
    \label{fig-b}
  \end{subfigure}
  \hfill
  \begin{subfigure}[t]{0.48\textwidth}
    \centering
    \includegraphics[width=\linewidth]{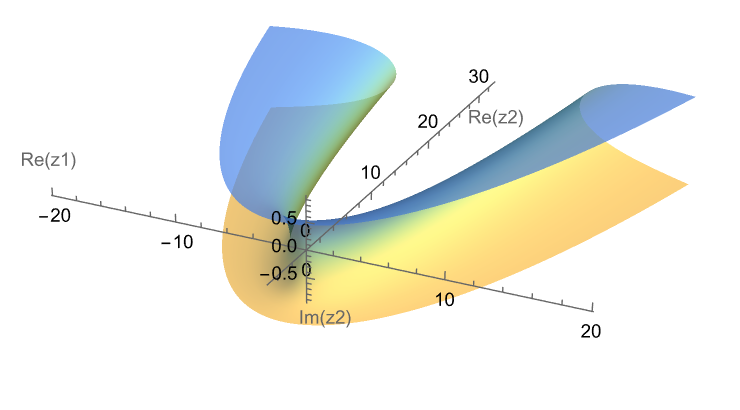}
    \caption{}
    \label{fig-e}
  \end{subfigure}
  \caption{The profile surfaces $\check F_{(b)}$ (equation \eqref{eq-Fb}) and $\check F_{(e)}$ (equation \eqref{eq-Fe}). They produce $G$-invariant special Lagrangians and Lagrangian translators (respectively) in $\mathbb{C}^3$ via Ansatz \ref{ans-1}. For the surface $F_{(e)}$, the choice of expander profile curve $\tilde \epsilon$ is shown in Figure \ref{fig-expander}. The colour of the surfaces corresponds to the argument of the first coordinate $z_1$ in both cases; the slight difference in colour range reflects the difference in the angles spanned by the Lawlor neck profile curve and Anciaux expander profile curve. }
\end{figure}

\begin{SCfigure}[0.6][h]  
  \centering
  \includegraphics[width=0.5\textwidth]{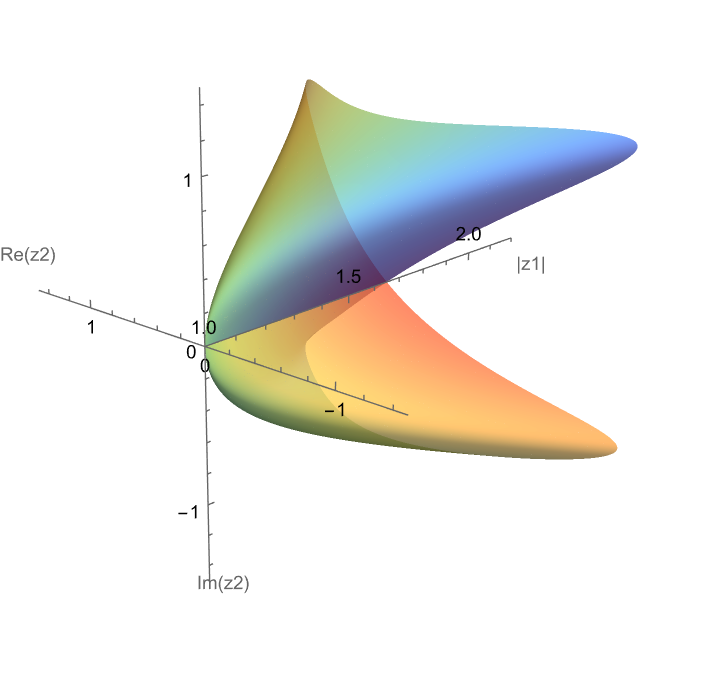}
  \caption{The profile surface $\check F_{(c)}$ (equation \eqref{eq-Fc}). It produces $G$-invariant special Lagrangians via Ansatz \ref{ans-1}. The colour of the surface corresponds to the argument of the first coordinate $z_1$ as in Figure 2. Though the surface appears immersed, one can see by the colour that it is an embedded surface in $\mathbb{C}^2$. 
  }
  \label{fig-c}
\end{SCfigure}

There are many other admissible subgroups $G\leq \SU(m-1)$ which may be used with this Ansatz, including several more infinite families of arbitrarily high dimension. A classification of admissible \textit{simple} subgroups $G$ is given in the table of Bedulli-Gori \cite[Table 1]{Bedulli2008} (where $n := m-1$). A non-simple example is given by the action of $\mathbb{S}^1 \times \SO(p) \times \SO(q)$ on $\mathbb{C}^{p+q}$ (see \cite[Example 1]{Bryant2004}, \cite[Section 7.3]{Wood2022} for discussion).

A full classification of the different special Lagrangians and translators produced by Ansatz \ref{ans-1} with a description of their properties is given in Sections \ref{sec-classification1} to \ref{sec-classification3}. In general, whether the examples produced are immersed, embedded or almost-calibrated depends on the choices of $G$, $\gamma$ and $\xi$. In particular, we may find $G$-invariant analogues of the embedded special Lagrangian of Example \ref{ex-c}, and the immersed translator of Example \ref{ex-f}. We also find $G$-invariant analogues of the translators of Examples \ref{ex-d} and \ref{ex-e} and the embedded special Lagrangian of Example \ref{ex-b}, although they have a one-dimensional singular set for $G \neq \SO(m-1)$. To the authors' knowledge, for $G \neq \SO(m-1)$, the translators and special Lagrangians produced by Ansatz \ref{ans-1} are all new examples. We note that our Ansatz does not produce any new complete, smoothly embedded, almost-calibrated translating solitons.

We also consider translators with symmetry groups other than linear subgroups of $\SU(m-1)$. In Section \ref{sec-affinegroups}, we show that symmetry groups of Lagrangian translators naturally lie in $\text{ASU}(n-1)$, the affine special unitary group. In Section \ref{sec-gaffine} we consider affine symmetry groups of the form $G \ltimes P$ for an isotropic subspace $P$, and show that Ansatz \ref{ans-1} applies in this case as well. In Section \ref{sec-moreexamples} we give examples of translators which are invariant under more exotic `screw-type' symmetry groups.  The simplest `screw-type' symmetry is the $\widetilde{U(1)}$-action on $\mathbb{C}^{m}$ given by
\begin{align*}
    \theta\longmapsto \left(e^{i\lambda_{1}\theta}z_1, \cdots, e^{i\lambda_{m-1}\theta}z_{m-1}, z_{m}-i\left(\textstyle\sum_{j=1}^{m-1}\lambda_{j}\right)\theta\right),
\end{align*}
and examples invariant under this kind of symmetry were constructed by Castro--Lerma \cite[Corollary 2]{Castro2010} and Konno \cite[Section 3.2]{Konno19}. We generalize the group action to $\widetilde{U(1)}^{m-1}$ acting on $\mathbb{C}^{m}$ (for demonstration we only consider $m=3$), and construct a family of Lagrangian translating solitons fibered over $\mathbb{R}^{m}$, with the central (singular) fiber being Hamiltonian stationary. Our examples are cohomogeneity-one with respect to their symmetry group, which contrasts with the non-existence result of Madnick--Wood for translators which are cohomogeneity-one with respect to a linear $G$-action (\cite[Theorem 5.10]{Wood2022}).

As well as the works already mentioned, constructions of novel Lagrangian translators include those of Joyce--Lee--Tsui \cite{Joyce2010}, Nakahara \cite{Nakahara13} and Su \cite{Su2019}, and Lagrangian mean curvature flow with symmetries has been previously studied by Pacini \cite{Pacini2002}, Savas-Halilaj-- Smoczyk \cite{Savas-Halilaj2018}, Konno \cite{Konno19}, Lotay--Oliveira \cites{Lotay2020, Lotay2023} Ochiai \cite{Ochiai2021}, Viana \cite{Viana2021}, Su \cite{Su2020} and Wood \cite{Wood2019}. 

\subsection*{Derivation of the Ansatz}

The Ansatz consists of two key parts: the immersion $\check F$ and the linear embedding $\Psi$; the Lagrangian is then given by the $G$-orbit of $\Psi \circ \check F$. Details of the construction of $\Psi$ are given in Section \ref{sec-4}.\footnote{In the notation of that chapter, $\Psi = (\Phi \oplus \text{Id})^{-1}$.} The key observation is that $G$-invariant Lagrangians lie in $\mu^{-1}(X) \times \mathbb{C}$ for some $X \in \mathfrak{g}$. The geometry of $\mu^{-1}(0)$ is special: it is the $m$-dimensional orbit of a complex line $P \subset \mathbb{C}^{m-1}$ under the $G$ action. Cohomogeneity-two Lagrangians $L \subset \mu^{-1}(0) \times \mathbb{C}$ are then characterised by the intersection $\check L := L \cap (P \times \mathbb{C})$. We choose the map $\Psi:\mathbb{C}^2 \to \mu^{-1}(0) \subset \mathbb{C}^m$ as the embedding of the complex space $P \times \mathbb{C}$ into $\mathbb{C}^m$, where we carefully identify $P$ with $\mathbb{C}$ to allow comparison of the Lagrangian angle. The special Lagrangian and translator conditions for $L$ then correspond to the PDEs \eqref{eq-cohom2translator}, \eqref{eq-cohom2SL} on $\check L$.

The construction of the Lagrangian map $\check F:I_1 \times I_2 \to \mathbb{C}^2$ is an adaptation of the Castro--Lerma Ansatz in \cite{Castro2010, CL14}, where the first coordinate $z_1$ is assumed to be the product of curves $\gamma(x)\xi(y)$, and the second coordinate is chosen so that $\check F$ satisfies the Lagrangian condition. Equations \eqref{eq-cohom2translator}, \eqref{eq-cohom2SL} then become the conditions \eqref{eq-curveeqsintro} of the Ansatz. A full derivation is  given in Section \ref{sec-ansatz}, in which Theorem \ref{thm-main} is stated precisely as Theorem \ref{thm-slagsandtranslators}.

\subsection*{Acknowledgments} This work was completed when the first named author was a postdoctoral researcher at National Center for Theoretical Sciences (NCTS), and during the second named author's postdoctoral research fellowships at King's College London and the Chinese University of Hong Kong, following a funded research visit to National Center for Theoretical Sciences. The first named author is grateful for the stimulating research environment provided by NCTS. The second named author was supported by a Simons Collaboration Grant on “Special holonomy in Geometry, Analysis and Physics", and by research grants from the Research Grants Council of the Hong Kong Special Administrative Region, China [Project No.: CUHK 14304121, CUHK 14305122 and CUHK 14307123]. We thank these institutions for their support.

\section{Preliminaries}

Throughout, we work in $\mathbb{C}^m$ for $m \geq 2$, equipped with its usual Calabi-Yau structure: the Euclidean metric $g_0$, the K\"ahler form $\omega_0 = \sum_{i=1}^m dx_i \wedge dy_i$, and the holomorphic volume form $\Omega_0 := dz_1 \wedge\ \ldots \wedge dz_m$.

\subsection{Lagrangian Translators}\label{sec-lagtrans}

Let $F:L\to\mathbb{C}^{m}$ be an oriented Lagrangian immersion, with unit volume form $\text{vol}_L$. The immersion is said to be \textit{zero-Maslov} if there exists a function $\theta: L \to \mathbb{R}$ satisfying $F^*(\Omega_0) = e^{i\theta}\text{vol}_L$. In this case, we have the differential relation $J\nabla \theta = \vec{H}$, where $\vec{H}$ is the mean curvature vector of the immersion $F$. 

$F$ is called a \textit{special Lagrangian submanifold} if $\vec{H} = 0$, equivalently if $\theta = \overline\theta$ for some $\overline\theta\in\mathbb{R}$. Such submanifolds are calibrated by the form $\text{Re}(e^{-i\overline\theta}\Omega_0)$, and therefore volume-minimising \cite{HL82}. If instead the range of $\theta$ is less than $\pi$, we say the Lagrangian is \textit{almost-calibrated}. 

$F$ is called  a \textit{Lagrangian translating soliton} (in the $e_m$ direction) if $\vec{H}  = e_{m}^{\perp}$ (where $e_{m} = (0, \cdots, 0, 1)$). The name comes from the fact that for such an $F$, the translating family of immersions $F_t := F + t e_m$ satisfies mean curvature flow. As in the special Lagrangian case, we may integrate to a scalar equation involving $\theta$:
\begin{align}\label{eq-translator}
    \theta(F) + \langle F, Je_{m}\rangle = c
\end{align}
for some $c\in\mathbb{R}$. By a translation in the $Je_{m}$-direction, we may assume $c = 0$. Thus, the imaginary part of the last coordinate function $F_{m}$ is equal to $-\theta(F)$.

It was noted in \cite{Su2019a} that for $f = -2\langle F, e_m\rangle$, the weighted volume $e^{-\frac{f}{2}}\text{vol}_{L}$ of a Lagrangian translating soliton is calibrated by $\re\Omega_{f}$, where in our case
\begin{align}
    \Omega_{f} := e^{z_{m}}dz_{1}\wedge\cdots\wedge dz_{m} \, = \, e^{z_m} \Omega_0.
\end{align}
Hence, an equivalent condition for $F:L\to\mathbb{C}^{m}$ being a Lagrangian translating soliton is
\begin{align}
    F^{*}\omega = F^{*}\im\Omega_{f} = 0. \label{eq-translatorcalibration}
\end{align}

\subsection{Lie Group Actions}\label{sec-liegroupactions}

Consider a group action $G \circlearrowright M$ for $M$ a  manifold, say of dimension $n$. For a point $z \in M$, we denote the \textit{orbit} by $\mathcal{O}_z := \{g \cdot z \, : \, g \in G\}$, and the \textit{isotropy subgroup} or \textit{stabiliser} by $G_z = \Stab(z) := \{g \in G\,:\, g\cdot z = z\}$. For $X \in \mathfrak{g}$, we denote the \textit{fundamental vector field} associated to $X$ by $\rho(X)_z := \frac{\partial}{\partial t}|_{t=0} (\exp(-tX) \cdot z)$. 

We say $z \in M$ has \textit{isotropy type $(H)$} for a subgroup $H \leq G$ if the isotropy group $G_z$ of $z$ is conjugate to $H$. Given $H \leq G$, we denote the set of points in $M$ with isotropy type $H$ by $M_{(H)}$. We note that the orbit $\mathcal{O}_z$ is diffeomorphic to the quotient $G / G_z$, and in particular $\text{dim}(\mathcal{O}_z) = \text{dim}(G) - \text{dim}(G_z)$.

We will say $M$ is \textit{cohomogeneity-$k$} if every point of $M$ has the same isotropy type $(H)$ (i.e. $M = M_{(H)}$) and $\text{dim}(G/H) = n-k$. This implies that all orbits in $M$ have dimension $n-k$. We remark that our notion of cohomogeneity-$k$ is slightly stricter than that of some other authors, who may allow orbits of different types, and `singular orbits' of lower dimensions.

Now assume $(M, \omega)$ is symplectic, and $G \circlearrowright M$ is a compact Hamiltonian group action\footnote{For an introduction to Hamiltonian group actions and moment maps, see \cite{DaSilva2006}.} with moment map $\mu:M \to \mathfrak{g}$, where $\mathfrak{g}$ and $\mathfrak{g}^*$ are identified via an $\text{Ad}$-invariant inner product (e.g. $M = \mathbb{C}^m$, $G \leq U(m)$ and $\mu:\mathbb{C}^m \to \mathfrak{g}$ the standard moment map of equation \eqref{eq-standardmomentmap}). We have the following useful result regarding the level sets of the moment map.

\begin{theorem}[Principal Orbit Theorem for Moment Map Level Sets]\label{thm-principalorbitthm}
    Let $X \in \mathfrak{g}$, and let $Z \subset \mu^{-1}(X)$ be a connected component of the level set of $\mu$ at $X$.
    \begin{itemize}
        \item For each conjugacy class $(K)$ of $G$, $Z_{(K)}$ is a smooth submanifold of $M$.
        \item There exists a unique conjugacy class $(H)$ of $G$ such that $Z_{(H)}$ is open, connected and dense in $Z$.
        \item If $(K) \neq (H)$ and $Y$ is a connected component of $Z_{(K)}$, then $\dim(Y) < \dim(Z_{(H)})$.
    \end{itemize}
\end{theorem}
\begin{proof}
     The claims follow from the stratification of $Z / G$ by orbit types (see \cite[Theorem 2.1]{Sjamaar1991}). For the second claim in particular, see \cite[Remark 5.10]{Sjamaar1991}.
\end{proof}

\subsection{The Affine Unitary Group}

Of particular importance to this work is the \textit{affine unitary group}, defined as the semidirect product
\[ \text{AU}(m) = \U(m) \ltimes \mathbb{C}^n = \{ (A, t) \, : \, A \in \U(m), \, t \in \mathbb{C}^n\}\]
with group product $(A_1, t_1) \cdot (A_2, t_2) = (A_1 A_2, A_1 t_2 + t_1)$. 
$\text{AU}(m)$ has a natural action on $\mathbb{C}^m$, given by $(A, t)\cdot z := Az + t$.
We note that by the Mazur-Ulam theorem, any group action preserving $(g_0, \omega_0)$ must be a subgroup of $\text{AU}(m)$:

\begin{lemma}
    If $G \circlearrowright \mathbb{C}^m$ is an effective group action preserving $g_0$ and $\omega_0$, then $G$ is isomorphic to a subgroup of $\emph{AU}(m)$, and the group action is induced by this isomorphism.
\end{lemma}

\section{Affine Group Actions Preserving $\Omega_f$}\label{sec-affinegroups}

Our aim is to construct $G$-invariant Lagrangian translating solitons for affine groups $G \leq \text{AU}(m)$ acting on complex Euclidean space $\mathbb{C}^m$ via the natural action. In light of equation \eqref{eq-translatorcalibration}, it is natural to restrict to subgroups $G \leq \text{AU}(m)$ preserving the form $\Omega_f$. In this section, we investigate such subgroups. Here are a couple of motivational examples:
\begin{example}
    The special orthogonal group $\SO(m-1)$ acts on $\mathbb{C}^m$ by acting on $\mathbb{C}^{m-1}$ in the standard way and by acting trivially on the final factor $\mathbb{C}$. Since $\SO(m-1) \leq \SU(m-1)$, and since it acts trivially on the $m$th factor, it is clear that this action preserves $(g_0, \omega_0, \Omega_f)$.
\end{example}
\begin{example}\label{ex-um-1}
    The universal cover of $\U(m-1)$ is $\SU(m-1) \times i\mathbb{R}$, with projection
\begin{align*}
    \pi:\widetilde{\U(m-1)} &\rightarrow \U(m-1)\\
    (B, i\theta) &\mapsto e^{i\theta}B.
\end{align*} 
Let $B: \widetilde{\U(m-1)}\to \SU(m-1)$, $\theta:\widetilde{\U(m-1)}\to \mathbb{R}$ be the projections. For any subgroup $G \leq \widetilde{\U(m-1)}$, we get an affine action $G \circlearrowright \mathbb{C}^{m}$:
\[ g \cdot (z', z_m) := \left(e^{i\theta(g)}B(g)\cdot z', \, z_m - i(m-1)\theta(g)\right).\]
It can be easily shown that such group actions preserve $(g_0, \omega_0, \Omega_f)$. The action of $\SO(m-1)$ in Example 1 is a special case of this group action.
\end{example}

The second example shows that certain `twisted' group actions are allowed, where a linear action of the first $(m-1)$ coordinates is paired with a translation in the final coordinate. We now generalise this observation, and find the most general subgroup of $\text{Aut}(\mathbb{C}^m)$ preserving $(g_0, \omega_0, \Omega_f)$. Recall that an automorphism $\phi:\mathbb{C}^m \to \mathbb{C}^m$ preserves $\Omega_f$ if for all $z \in \mathbb{C}^m$ and all $X_1, \ldots, X_m \in T_z\mathbb{C}^m$,
\[  (\Omega_f)_z(X_1, \ldots, X_m) = (\Omega_f)_{\phi(z)}(\phi_*(X_1), \ldots, \phi_*(X_m)). \]
In the case of $(A, t) \in \text{AU}(m)$, this is equivalent to 
\begin{align*}
    e^{z_m}\Omega(X_1, \ldots, X_m) \, &= \, (\Omega_f)_{Az + t}((A,t)_*X_1, \ldots, (A,t)_*X_m)\\
    &= \, e^{(Az + t)_m}\det(A)\,\Omega(X_1, \ldots, X_m)\\
    \iff e^{(Az + t)_m} \det(A) &= e^{z_m}.
\end{align*} 
By considering several values of $z \in \mathbb{C}^m$, this is equivalent to 
\begin{align*}
\begin{cases}
    (Az)_m = z_m ,\\
    e^{t_m} = \text{det}(A)^{-1}.
\end{cases}
\end{align*}
We can therefore conclude that 
\begin{align*}
        A = \begin{pmatrix}
            B & b \\ 0 & 1
        \end{pmatrix}, \,\,\,\,
        t = \begin{pmatrix}
            t' \\
            - i \theta
        \end{pmatrix},
\end{align*}
for some $B \in \U(m-1)$, $b, \,t' \in \mathbb{C}^{m-1}$, and $\theta \in \mathbb{R}$ satisfying $e^{i\theta} = \det(A) = \det(B)$. To sum up, we have the following result:

\begin{proposition}\label{prop-affinegrouppreservingomegaf}
The subgroup of $\emph{AU}(m)$ that preserves $\Omega_f$, which we denote $G_f$, is given by:
\begin{align*}
       G_f \, = \, \bigg\{ \bigg(
         \begin{pmatrix}
            B & b \\ 0 & 1
        \end{pmatrix}&, \,\begin{pmatrix}
            t' \\
            -i\theta  
        \end{pmatrix} \bigg) \in AU(m) \,:\\
        &B \in \U(m-1), \, b,\,t' \in \mathbb{C}^{m-1}, \, \theta \in \mathbb{R}, \,e^{i\theta}= \det(B) \bigg\},
\end{align*}
with action on $\mathbb{C}^{m}$ induced by the action of $AU(m)$:
\begin{align}
    \left(
    \begin{pmatrix}
        B & b \\ 0 & 1
    \end{pmatrix}, \,\begin{pmatrix}
        t' \\
        -i\theta 
    \end{pmatrix} \right) \cdot
    \begin{pmatrix}
        z'\\ z_m
    \end{pmatrix}
    =
    \begin{pmatrix}
    B z' + b z_m + t' \\ z_m - i \theta
    \end{pmatrix}.\label{eq-gfaction}
\end{align}
\end{proposition} 

\begin{example}
The group $\widetilde{\U(m-1)}$ considered in Example \ref{ex-um-1} is precisely the subset of $G_f$ for which $b=0,\, t'=0$, with the corresponding group action on $\mathbb{C}^m$.
\end{example}

If $G \leq G_f$ is the symmetry group of an \textit{almost-calibrated} Lagrangian translating soliton $L$ with no linear factor (i.e. $L$ is not a Riemannian product of $L'$ with a line), then a further restriction can be made:

\begin{proposition} Let $G \leq G_f$, and let $L \subset \mathbb{C}^{m}$ be a $G$-invariant, almost-calibrated, Lagrangian translating soliton with no linear factor. Then $G \leq \emph{ASU}(m-1)$.
\end{proposition}\label{prop-acadmissiblegroup}
\begin{proof}
    Since $L$ is almost-calibrated, we have from (\ref{eq-translator}) that the imaginary part of the $z_m$ coordinate is bounded on both sides, i.e. there exist $c, C$ such that $L \subset \{z \in \mathbb{C}^m \, : \, c < \text{Im}(z_m) < C\}.$ It follows from the expression (\ref{eq-gfaction}) for the group action that $\theta = 0$ for all $g \in G$. Furthermore, since $L$ does not split off a line, we can conclude that $G$ preserves the translation direction $e_m$, and so $b = 0$ for all $g \in G$. 
\end{proof}

\begin{remark}
    If $L$ is a special Lagrangian, then it may be shown that the connected symmetry subgroup $\emph{Sym}^0(L) \leq \emph{AU}(m)$ is a subgroup of $\emph{ASU}(n)$. It would be desirable to establish the analogous result for Lagrangian translators, namely that if $L$ is a Lagrangian translator then $G:=\emph{Sym}^0(L) \leq G_f$. It may be proven using Lie group structure theory that if $L$ is a Lagrangian translator with no linear component, and if the elements of $G$ have no translational component in the direction of $e_m$, then this is the case.
\end{remark}

\section{$G$-invariant Lagrangian Submanifolds}\label{sec-4}

We now develop structure theory for $G$-invariant Lagrangian submanifolds in $\mathbb{C}^m$ for $m \geq 2$. With Section \ref{sec-affinegroups} in mind, we restrict to $G \leq \text{ASU}(m-1)$.

\subsection{Symmetric Lagrangians with Linear Symmetry Group}\label{sec-glinear}

We first consider the linear case. Consider a compact, connected subgroup $G \leq \SU(m-1)$ along with its natural action on $\mathbb{C}^{m-1}$. We will show that a cohomogeneity-two Lagrangian $L \subset \mathbb{C}^n$ is characterised by the surface $\check L$ given by intersection with a `complex profile plane', and we provide conditions on $\check L$ which equate to $L$ being special Lagrangian or a translating soliton for mean curvature flow. The results of this section adapt the work of Madnick--Wood \cite{Wood2022}. 

We first make the key observation that $G$-invariant Lagrangian submanifolds are contained in level sets of the moment map; indeed in the \textit{principal part} of those level sets.

\begin{proposition}[$G$-invariant Lagrangians are contained in $\mu$ level sets]\label{prop-laginlevelset}
    Let $G \leq \SU(m-1)$ be compact and connected, and consider the action $G \circlearrowright \mathbb{C}^{m-1}$. Let $\mu:\mathbb{C}^{m-1}\to \mathfrak{g}$ be the standard moment map, and let $L$ be a closed, immersed Lagrangian submanifold. Then:
    \begin{itemize}
        \item[a)] $L$ is $G$-invariant if and only if there exists $X \in \mathfrak{g}$ fixed by the coadjoint action such that $L \subset \mu^{-1}(X)$.
        \item[b)] $L$ is $G$-invariant of cohomogeneity-$k$ and isotropy type $(H)$ if and only if there exists $X \in \mathfrak{g}$ fixed by the coadjoint action and a connected component $Z$ of $\mu^{-1}(X)$ such that $Z_{(H)}$ is open, connected and dense in $Z$, $L \subset Z_{(H)}$ and $\dim(Z) = \dim(Z_{(H)}) =  m-1+k$.
    \end{itemize}
\end{proposition}

\begin{proof}
    For part a), the forward direction follows from \cite[Proposition 4.2]{Joyce2001a}, and the converse is \cite[Corollary 3.13]{Wood2022}. For part b), the reverse direction follows from \cite[Corollary 3.15]{Wood2022}. For the forward direction, consider a cohomogeneity-$k$ Lagrangian $L\subset \mathbb{C}^{m-1}_{(H)}$. By part a) there exists $X \in \mathfrak{g}$ and a connected component $Z$ of $\mu^{-1}(X)$ such that $L \subset Z$. It remains to show that $(H)$ is the principal orbit type of $Z$, in the sense of Theorem \ref{thm-principalorbitthm}.

    Let $Y$ be the connected component of $Z_{(H)}$ containing $L$, and assume for a contradiction that $(H') \neq (H)$ is the principal orbit type.  By Theorem \ref{thm-principalorbitthm}, $\dim(Y) < \dim(Z_{(H')})$. There are now two options. If $(H)$ is an exceptional orbit type, i.e. $\dim(G/H) = \dim(G/H')$, then by the decomposition (18) of \cite{Wood2022},we would have $\dim(\mathcal{H}_z)<k$. This contradicts \cite[Corollary 3.15]{Wood2022}. On the other hand, if $(H)$ is a singular orbit type, then $\dim(G/H') = m-1-l$ for $l < k$. Now by \cite[Remark 3.10]{Wood2022} and \cite[Corollary 3.15]{Wood2022}, $\dim(Z_{(H')}) \leq m-1 + l < m-1 + k = \dim(Y)$, which gives a contradiction.
\end{proof}

The geometry of the zero level set $\mu^{-1}(0)$ is particularly simple, as the next proposition shows, and we will restrict to this setting for the remainder of the paper. As well as being convenient, the restriction is natural, and occurs in several common settings. For example, if $L$ is an \textit{exact} almost-calibrated cohomogeneity-one $G$-invariant Lagrangian (as might occur as a blow-up model for Lagrangian mean curvature flow, see \cite[Theorem 2.12]{Joyce2015}), it must lie in $\mu^{-1}(0)$ by \cite[Proposition 4.10]{Wood2022}. In fact, if $G$ is semisimple, then \textit{all} $G$-invariant Lagrangians lie in $\mu^{-1}(0)$ by Proposition \ref{prop-laginlevelset} part a).

\begin{proposition}[Geometry of the zero level set]\label{prop-levelsetgeometry}
    Let $G \leq \SU(m-1)$ be compact and connected, and consider the action $G \circlearrowright \mathbb{C}^{m-1}$. Let $\mu:\mathbb{C}^{m-1} \to \mathfrak{g}$ be the standard moment map. Let $Z$ be a connected component of $\mu^{-1}(0)$, and assume that $Z$ contains a cohomogeneity-one $G$-invariant Lagrangian submanifold $L$, so that $\text{dim}(Z) = m$. Then:
    \begin{itemize}
        \item[a)] $L \subset Z \setminus \{0\}$.
        \item[b)] Given $z \in Z\setminus\{0\}$, and denoting by $P_z$ the complex line through $z$, we have \[ Z \, = \, \{g\cdot \lambda z \, :\, g \in G, \, \lambda \in \mathbb{C}\} \, = \, G \cdot P_z.\]
        \item[c)] There exists a finite cyclic group $C_k$ defined as $C_k := \Stab(P_z)/\Stab(z)$ with order $k|(2m-2)$ acting in the natural way on $P_z$ such that $L \cap P_z$ is $C_k$-invariant.
    \end{itemize}
\end{proposition}

\begin{proof}
    By Proposition \ref{prop-laginlevelset}, $L\subset Z_{(H)}$ for $H$ the principal orbit type. In particular, since $\{0\}$ is a singular orbit, $L \subset Z \setminus \{0\}$.
    Part b) follows from Proposition 4.11 a) of \cite{Wood2022}, and part c) follows from Propositions 4.6 and 4.11 c) of the same.
\end{proof}

The following two examples are the easiest to describe explicitly, although there are many other subgroups $G \leq \SU(m-1)$ admitting cohomogeneity-one Lagrangians in $\mu^{-1}(0) \subset \mathbb{C}^{m-1}$. 

\begin{example}\label{ex-so(m-1)}
    Consider $\SO(m-1) \subset \SU(m-1)$ acting diagonally on $\mathbb{C}^{m-1} \cong \mathbb{R}^{m-1} \oplus \mathbb{R}^{m-1}$. The zero level set can be calculated to be:
    \begin{align*}
        \mu^{-1}(0) \, &= \, \{ z \in \mathbb{C}^{m-1} \, : \, z_i \overline z_j = \overline z_i z_j \, \text{ for } 1 \leq i, j \leq m-1\}\\
        &= \, \{ (x,y) \in \mathbb{R}^{m-1} \oplus \mathbb{R}^{m-1} \, : \, \{x,y\} \text{ linearly dependent}\}\\
        &= \, \SO(m-1) \cdot P_z
    \end{align*}
    where $z = (1,0,\ldots,0).$ For this choice of $z$, we have
    \[ \emph{Stab}(z) \cong \SO(m-2), \quad \emph{Stab}(P_z) \cong \emph{O}(m-2),\]
    so that $C_k = \emph{O}(m-2)/\SO(m-2) = C_2$. Note that $\SO(m-1)$ is a simple group, and so by Proposition \ref{prop-laginlevelset}, any $G$-invariant Lagrangian $L$ must lie in $\mu^{-1}(0)$, and $l := L \cap P_z$ is a $C_2$-invariant curve.
\end{example}
\begin{example}\label{ex-t(m-2)}
    Consider the set of diagonal elements $T^{m-2} \subset \SU(m-1)$. In this case,
    \begin{align*}
        \mu^{-1}(0) \, &= \, \{ z \in \mathbb{C}^{m-1} \, : \, |z_i|^2 = |z_j|^2\text{ for } 1 \leq i, j \leq m-1\}\\
        &= \, T^{m-1} \cdot P_z
    \end{align*}
    where $z = (1,1,\ldots,1).$ For this choice of $z$, we have
    \[ \emph{Stab}(z) = \{\emph{Id}\}, \quad \emph{Stab}(P_z) = \{\emph{diag}(e^{\frac{2k\pi i}{m-1}}, \ldots, e^{\frac{2k\pi i}{m-1}}) \, : \, k \in \mathbb{Z}\},\]
    so that $C_k = C_{m-1}$. $T^{m-2}$ is abelian, and unlike for $G = \SO(m-1)$, \emph{every} level set $\mu^{-1}(X)$ of $\mu$ admits cohomogeneity-one Lagrangian submanifolds. The other level sets do not admit such a simple description as $\mu^{-1}(0)$.
\end{example}

The following proposition considers the case where $G \circlearrowright \mathbb{C}^{m-1}$ is extended to act on $\mathbb{C}^m$, and gives an expression for the structure of $\mu^{-1}(0)$.

\begin{proposition}\label{prop-cohom1reduction}
    Let $G \leq \SU(m-1)$ be compact and connected. Extend the action $G \circlearrowright \mathbb{C}^{m-1}$ to an action on $\mathbb{C}^m$ by acting trivially on the final coordinate. Denote by $\mu_{\mathbb{C}^m}:\mathbb{C}^m \to \mathfrak{g}$ the standard moment map for the action of $G$ on $\mathbb{C}^m$. Then 
    \begin{equation*}
        \mu_{\mathbb{C}^m}^{-1}(0) = \mu^{-1}(0) \times \mathbb{C}.
    \end{equation*}
    \end{proposition}

\begin{proof}
    This is an immediate consequence of the definition of the standard moment map.
\end{proof}

From this point on, we focus on cohomogeneity-two Lagrangians in $\mu_{\mathbb{C}^m}^{-1}(0)$. For the remainder of the section, we will therefore assume we are in the following setting:
\begin{assumption}\label{as1}
    Consider a compact, connected subgroup $G \leq \SU(m-1)$ acting on $\mathbb{C}^{m-1}$ and a connected component $Z$ of $\mu^{-1}(0)$ admitting cohomogeneity-one Lagrangian submanifolds. Extend to an action on $\mathbb{C}^m$ by acting trivially on the final coordinate. Choose $z' \in Z \setminus \{0\}$.
    
    Define $P := P_{z'}$ to be the complex line through $z'$ in $\mathbb{C}^{m-1}$, so that $P \times \mathbb{C}$ is then a Kähler subspace of $\mathbb{C}^m = \mathbb{C}^{m-1}\times \mathbb{C}$. Define $C_k$ to be the finite cyclic group acting on $P$ as in Proposition \ref{prop-levelsetgeometry}.
\end{assumption}

From equations (21) and (23) of \cite{Wood2022}, along with Proposition \ref{prop-cohom1reduction}, we get the following orthogonal decompositions of tangent spaces for any $z_m \in \mathbb{C}$:
\begin{align}
    T_{(z',z_m)}\mathbb{C}^m \, &= \, T_{z'} \mathcal{O}_{z'} \oplus JT_{z'}\mathcal{O}_{z'} \oplus P \oplus \mathbb{C} \label{eq-ambientspacedecomposition},\\
    T_{(z',z_m)} \mu^{-1}_{\mathbb{C}^m}(0) \, &= \, T_{z'}\mathcal{O}_{z'} \oplus P \oplus \mathbb{C}. \label{eq-levelsettangentdecomposition}
\end{align}

The key fact about cohomogeneity-two Lagrangians in $\mu^{-1}(0)$ is that they are characterised by the intersection $\check{L} := L \cap (P \times \mathbb{C})$, which is a 2 dimensional Lagrangian submanifold.

\begin{proposition}\label{prop-correspondence}
    In the setting of Assumption \ref{as1}, there is a 1-1 correspondence between:
    \begin{itemize}
        \item $G$-invariant embedded/immersed cohomogeneity-two Lagrangian submanifolds $L \subset Z \times \mathbb{C} \subset \mathbb{C}^m$,
        \item $C_k$-invariant 2-dimensional embedded/immersed Lagrangian submanifolds $\check{L} \subset (P \setminus \{0\}) \times \mathbb{C} \subset \mathbb{C}^{m-1}\times \mathbb{C}$ ,
    \end{itemize}
    where $P = P_{z'}, C_k$ are as in Proposition \ref{prop-levelsetgeometry} parts b) and c) respectively. The map $L \mapsto \check{L}$ is given by intersection with $P \times \mathbb{C} \subset \mathbb{C}^{m-1}\times\mathbb{C}$, and the map $\check{L} \mapsto L$ is given by $L := G \cdot \check{L}$.
    
\end{proposition}

\begin{proof}
    The proof follows exactly as in \cite[Proposition 4.12]{Wood2022}. We outline one direction.
    Consider a $G$-equivariant cohomogeneity-two immersion $F:\hat{L} \to \mu^{-1}(0) \subset \mathbb{C}^m$; we show that $F|_{F^{-1}(P \times \mathbb{C})}$ is a 2-dimensional $C_m$-equivariant Lagrangian immersion. 

    Choose $p \in F^{-1}(P \times \mathbb{C})$, and an open neighbourhood $U \ni p$ in $\hat{L}$ such that $F(U)$ is an embedded Lagrangian. By (\ref{eq-levelsettangentdecomposition}), the intersection of $T_{F(p)}F(U)$ with $P \times \mathbb{C}$ is transverse in $\mu^{-1}(0)$, and so $F(U) \cap (P\times \mathbb{C})$ is an embedded submanifold of dimension 2. It follows that $F|_{F^{-1}(P\cap\mathbb{C})}$ is a Lagrangian immersion. The $C_k$-invariance follows from the definition of the $C_k$ action on $P$ (see \cite[Proposition 4.6]{Wood2022}).
\end{proof}

Furthermore, we may express the Lagrangian angle of a cohomogeneity-two Lagrangian $L$ in terms of the Lagrangian angle of $\check{L}\subset P \times \mathbb{C}$. For the purpose of defining the latter Lagrangian angle, we identify the 2-dimensional Kähler subspace $P \times \mathbb{C}$ with $\mathbb{C}^2$ (with its standard Calabi-Yau structure) using an isomorphism $\Phi$.

\begin{proposition}\label{prop-lagangle}
    In the setting of Assumption \ref{as1}, there exists an isomorphism $\Phi:P \to \mathbb{C}$ such that the following is true.

    Consider a $C_k$-invariant immersed Lagrangian submanifold $\check{L} \subset P \times \mathbb{C}$, corresponding to a cohomogeneity-two Lagrangian $L := \{g\cdot w \, : \, w \in \check{L}\}$, and denote, for $w = (w', w_m) \in \check{L}$:
    \begin{itemize}
        \item $\theta_L (g \cdot w)$ the Lagrangian angle of $L \subset \mathbb{C}^m$ at $g \cdot w$,
        \item $\theta_{\check{L}}(w)$ the Lagrangian angle of $(\Phi\oplus \text{Id})(\check{L}) \subset \mathbb{C}^2$ at $(\Phi\oplus \text{Id})(w)$,
        \item $\arg_{\Phi}(w')$ the argument of $\Phi(w')$ in $\mathbb{C}$.
    \end{itemize}
    Then for $w = (w', w_m) \in L$ and $g \in G$:
    \[ \theta_L(g\cdot w) \equiv \theta_{\check{L}}(w) \, + \, (m-2)\arg_{\Phi}(w') \quad (\emph{mod } \pi).\]
\end{proposition}

\begin{proof}
    For a point $v \in P \times \mathbb{C}$, we embed $P \times \mathbb{C} \subset \mathbb{C}^{m-1}\times \mathbb{C}$ as usual, and consider the orbit $\mathcal{O}_v$ of $v$ in $\mathbb{C}^m$ of dimension $m-2$. Let $\pm\nu_{v}$ denote the pair of unit elements of $\Lambda^{m-2}T_v\mathcal{O}_v$.

    Noting that $(e^{i\phi})^*\Omega = e^{im\phi}\Omega$, it follows that we may choose $V \in P$ such that $\{\Omega_V(\pm\nu_V, V, e_m)\} = \{\pm 1\}$. Define $\Phi:P \to \mathbb{C}$ to be the unique complex linear map such that $\Phi(V) = 1$. This implies that, for $X, Y \in P \times \mathbb{C}$, 
    \begin{equation}\label{eq-omegaeq1}
        \{\Omega_V(\pm \nu_V, X, Y)\} \, = \, \{\pm \Omega_0((\Phi\oplus\text{Id})X,(\Phi\oplus\text{Id})Y)\},
    \end{equation}
    where $\Omega_0$ denotes the standard holomorphic volume form on $\mathbb{C}^2$. Choose a point $w=(w', w_m) \in \check{L}$, and choose an orthonormal basis $\{X_1, X_2\}$ of $T_w\check{L}$, so that
    \begin{equation}\label{eq-omegaeq2}
        e^{i \theta_{\check{L}}(w)} = \Omega_0((\Phi\oplus\text{Id})X_1, (\Phi\oplus\text{Id})X_2).
    \end{equation}
    Then, writing $w' = re^{i\alpha}$, and using (\ref{eq-omegaeq1}) and (\ref{eq-omegaeq2}) for the final equality:
    \begin{align*}
        \pm e^{i\theta_L(w)} \, &= \, \Omega_{re^{i\alpha}V}(\pm \nu, X_1, X_2)\\
        &= \, ((e^{i\alpha})^*\Omega)_{rV}(\pm \nu, e^{-i\alpha} X_1, e^{-i\alpha}X_2)\\
        &=e^{i(m-2)\alpha}\,\Omega_{rV}(\pm \nu, X_1, X_2)\\
        &= \pm e^{i(m-2)\alpha\, + \, \theta_{\check{L}}(w)}
    \end{align*}
    (where all equalities are of sets). This implies the result for $g=e$; the general case follows since $L$ is $G$-invariant and $G$ preserves the form $\Omega$.
\end{proof}

We conclude in the following corollary that the study of cohomogeneity-two translating solitons is reduced to studying certain 2-dimensional Lagrangian submanifolds in $\mathbb{C}^2$.

\begin{corollary}[Characterisation of cohomogeneity-two translators/special Lagrangians]\label{cor-cohom2translators}
    In the setting of Assumption \ref{as1}, there is a 1-1 correspondence between:
    \begin{itemize}
        \item $G$-invariant embedded/immersed cohomogeneity-two Lagrangian translating solitons $L \subset Z \times \mathbb{C} \subset \mathbb{C}^m$,
        \item $C_k$-invariant embedded/immersed Lagrangian submanifolds $\check{L}\subset \mathbb{C}^2$, satisfying the equation
        \begin{equation}\label{eq-cohom2translator} \theta_{\check{L}}(z_1, z_2) + (m-2)\emph{arg}(z_1) + \emph{Im}(z_2)\, \equiv \, 0 \quad (\emph{mod }\pi).
        \end{equation}
    \end{itemize}
    Furthermore, there is a 1-1 correspondence between:
    \begin{itemize}
         \item $G$-invariant embedded/immersed cohomogeneity-two special Lagrangians $L \subset Z \times \mathbb{C} \subset \mathbb{C}^m$ (with Lagrangian angle $0$),
        \item $C_k$-invariant embedded/immersed Lagrangian submanifolds $\check{L}\subset \mathbb{C}^2$, satisfying the equation
        \begin{equation}\label{eq-cohom2SL}
        \theta_{\check{L}}(z_1, z_2) + (m-2)\emph{arg}(z_1) \, = \, 0.
        \end{equation}
    \end{itemize}
    In both cases, the map $L \mapsto \check{L}$ is given by intersection with $P \times \mathbb{C} \subset \mathbb{C}^{m-1}\times\mathbb{C}$, and the map $\check{L} \mapsto L$ is given by $L := G \cdot \check{L}$.
\end{corollary}

\subsection{Symmetric Lagrangians with Translational Symmetries}\label{sec-gaffine}

We briefly consider more general affine symmetry groups $G \leq \text{ASU}(m-1)$. We show we may reduce to the case where $G$ contains no one-parameter family of pure translations, equivalently where $L$ has no linear factor.

\begin{proposition}\label{prop-translationreduction}
    Consider a connected Lagrangian $L \subset \mathbb{C}^{m-1}$ with connected symmetry group $G := \emph{Sym}^0(L) \leq \emph{ASU}(m-1)$ with Lie algebra $\mathfrak{g} := \emph{Lie}(G) \leq \mathfrak{asu}(m-1)$. Let $P := \exp(\mathfrak{g} \cap (0 \times \mathbb{C}^{m-1}))$ be the identity component of the subgroup of pure translational symmetries, and denote $d := \dim(P)$. Then:

    \begin{itemize}
        \item $P$ is  an isotropic subspace of $\mathbb{C}^{m-1}$. $G$ may be expressed as $G' \ltimes P$, where $G'$ preserves the subspace $\mathbb{C}^{m-1-d} := (P \oplus JP)^\perp$.
        \item Up to a translation by $w \in JP$, $L = P \times L'$ for $L' \subset (P \oplus JP)^\perp \cong \mathbb{C}^{m-1-d}$ a Lagrangian submanifold. 
        \item $\check L$ is invariant under the restricted group $G'|_{\mathbb{C}^{m-1-d}} \leq \emph{ASU}(m-1-d)$. This contains no pure translations, in the sense that $\emph{Lie}(G'|_{\mathbb{C}^{m-1-d}}) \cap (0\times \mathbb{C}^{m-1-d}) = 0.$  Furthermore, if $L$ is cohomogeneity-$k$ with respect to $G \ltimes P$, then $L'$ is cohomogeneity-$k$ with respect to $G'$.
    \end{itemize}
\end{proposition}
\begin{proof}

    Since $L$ is $G$-invariant, it follows that $P$ is isotropic. The action of $P$ has moment map $\mu_P:\mathbb{C}^{m-1} \to P^*$, $\langle \mu_P(z), \, p\rangle = \langle -iz, \, p\rangle.$ Since $L$ is invariant under the action of $P$, it follows from \cite[Corollary 4.4]{Joyce2001a} that there exists $\xi \in P
    ^*$ such that $L \subset \mu_P^{-1}(\xi)$, which means that the projection of $L$ to $JP$ is constant. Translating by this constant, we may assume that $L = \check L \times P \subset \mathbb{C}^{m-1-d} \times P$, for $\mathbb{C}^{m-1-d} := (P\oplus JP)^\perp$.

    Now, we may write $G = G' \ltimes P$, where the elements of $G'$ are of the form $(g, v)$ for $g \in \SU(m-1)$, $v \in \mathbb{C}^{m-1-d}$. Note that for $(g,v) \in G'$ and $(\text{Id}, p) \in P$,
    \[ (g, v)\cdot (\text{Id}, p) \cdot (g, v)^{-1} \, = \, (\text{Id}, g\cdot p) \, \in \, P,\]
    so $G$ preserves the orthogonal decomposition $P \oplus JP \oplus \mathbb{C}^{m-1-d}$. It follows that $\check L$ is invariant under the group $G' |_{\mathbb{C}^{m-1-d}}$. Finally, $G'|_{\mathbb{C}^{m-1-d}}$ cannot have any pure translations, else $\check L$, and therefore $L$, would have those same translational symmetries, and so they would have been elements of $P$.
\end{proof}

\begin{remark} By the above proposition we may extend the structure theory of Section \ref{sec-glinear} to apply to Lagrangians with symmetry groups of the form $G' \ltimes P$ where $G'\leq \SU(m-1)$ and $P$ is a $G'$-invariant isotropic subspace, by applying it to $L'$ with symmetry group $G'|_{\mathbb{C}^{m-1-d}}$.
\end{remark}

\begin{remark}
    We note that there are affine subgroups that are not of the form $G \ltimes P$, i.e. `screw-type' subgroups that combine linear and translational movements. Example \ref{ex-um-1} gives one such subgroup, and we consider examples of Lagrangian translators invariant under such subgroups in Section \ref{sec-moreexamples}. The structure of the moment map level sets is different to those of linear subgroups, and the theory of Section \ref{sec-glinear} does not apply.
\end{remark}

\section{Cohomogeneity-Two Translators and Special Lagrangians}\label{sec-5}

We now use Corollary \ref{cor-cohom2translators} together with a generalisation of a construction by Castro--Lerma to construct new cohomogeneity-two Lagrangian translators and special Lagrangians.

\subsection{Definition of the Ansatz}\label{sec-ansatz}

Let $I\subseteq\mathbb{R}$ be an open interval and $s_{0}\in I$. Let $\gamma:I\to\mathbb{C}$ be a smooth curve, $\gamma(s) = x(s) + iy(s)$, $x(s), y(s)\in\mathbb{R}$. Consider the function $\beta_\gamma: I \to \mathbb{R}$, defined by
\begin{equation}
    \beta_{\gamma}(s) \, := \,  -\int_{s_{0}}^{s}\langle\gamma(t), i\gamma'(t)\rangle\:dt.
\end{equation}
Then if $\lambda := \frac{1}{2i}(\overline z dz - zd\overline{z})$ denotes the standard Liouville form on $\mathbb{C}$ (satisfying $d\lambda = 2\omega_0$), we have:
\begin{align*}
    \gamma^*(\lambda) = (x(s)y'(s) - y(s)x'(s))\:ds = \langle\gamma(s), -i\gamma'(s)\rangle\:ds = \frac{d\beta_{\gamma}}{ds}\:ds,
\end{align*}
i.e. $\beta_{\gamma}$ is the primitive of $\gamma
^*(\lambda)$. We also define the related integral quantity
\begin{equation}
    \widetilde \beta_\gamma(s) \, := \, -\int_{s_0}^s |\gamma^2(t)|\, \langle \gamma(t), \, i\gamma'(t)\rangle \: dt,
\end{equation}
which satisfies $d\widetilde \beta_\gamma(s) = |\gamma(s)|^2\frac{d\beta_\gamma}{ds}ds \, = \, |\gamma(s)|^2 \, \gamma^*(\lambda)$, i.e. $\widetilde \beta_\gamma$ is the primitive of $|\gamma|^2\gamma^*(\lambda)$.

\begin{remark}\label{rem-betapositivity}
    We note that if a curve $\gamma$ is parametrised in polar coordinates $\gamma = r(s)e^{is}$, we have
\begin{align*}
    \gamma'(s) &= (r'(s) + ir(s))e^{is},
    \,\,\,\, \beta_{\gamma}(s) \, = \, \int_{s_0}^s |r(t)|^2\, dt.
\end{align*}
In particular, $\beta_\gamma(s)$ is strictly increasing in $s$.
\end{remark}

The following Ansatz, based on that of Castro and Lerma \cite{Castro2010}, produces a Lagrangian surface from two curves $\gamma$, $\xi \subset \mathbb{C}$.

\begin{proposition}\label{prop-lagrangianansatz}
    Let $I_{1}, I_{2}\subseteq\mathbb{R}$ be open intervals, and let $\gamma: I_{1}\to\mathbb{C}$,  $\xi:I_{2}\to\mathbb{C}$ be smooth curves, at least one of which doesn't intersect $0$. Then the map $\check F:I_{1}\times I_{2}\to\mathbb{C}^{2}$ given by
    \begin{align}\label{eq-ansatz}
        \check F(x, y) = \left(\gamma(x)\xi(y), \frac{1}{2}(|\xi(y)|^{2} - |\gamma(x)|^{2}) + i(\beta_{\xi}(y) - \beta_{\gamma}(x)) \right)
    \end{align}
    is a Lagrangian immersion, with Lagrangian angle
    \begin{align}
        \check \theta(x, y) \, = \, \arg\left(\gamma'(x)\right)+\arg\left(\xi'(y)\right).
    \end{align}
    Furthermore, if $\lambda := \frac{1}{2i}\sum_{i=1}^2(\overline z_i dz_i - z_id\overline{z_i})$ denotes the standard Liouville form on $\mathbb{C}^2$, then 
    \begin{equation}\label{eq-restrictionofliouville}
        \check F^*(\lambda) = d\left( -\frac{1}{2}(|\gamma(x)|^2 - |\xi(y)|^2)(\beta_\gamma(x) - \beta_\xi(y)) + \widetilde \beta_\gamma(x) + \widetilde \beta_\xi(y) \right),
    \end{equation}
    i.e. the image of $\check F$ is an exact Lagrangian submanifold of $\mathbb{C}^2$.
\end{proposition}

\begin{remark}
    Up to translation in the second complex coordinate, the function $\check F$ is given by \[ \check F(x,y) \, = \, \left( \gamma(x)\xi(y), \, \int_0^y \overline {\xi(s)} \xi'(s)\, ds \, - \, \int_0^x \overline {\gamma(s)} \gamma'(s)\,ds \right).\]
\end{remark}

\begin{proof}
    Observe that by the definition of $\beta_{\xi}$,
    \begin{align*}
        \frac{d}{dy}\left(\frac{|\xi(y)|^{2}}{2} + i\beta_{\xi}(y)\right) &= \frac{\overline{\xi(s)}\xi'(s) + \overline{\xi'(s)}\xi(s)}{2} - i\langle\xi(s), i\xi'(s)\rangle\\
        &= \re\left(\overline{\xi(s)}\xi'(s)\right) + i\im\left(\overline{\xi(s)}\xi'(s)\right)\\
        &= \overline{\xi(s)}\xi'(s).
    \end{align*}
    The same holds for $\gamma(x)$. Thus,
    \begin{align*}
        &\frac{\partial \check F}{\partial x} = \left(\gamma'(x)\xi(y), -\gamma'(x)\overline{\gamma(x)}\right) = \gamma'(x)(\xi(y), -\overline{\gamma(x)}),\\
        &\frac{\partial \check F}{\partial y} = \left(  \gamma(x)\xi'(y), \xi'(y)\overline{\xi(y)}\right) = \xi'(y)(\gamma(x), \overline{\xi(y)}),
    \end{align*}
    so since one of $\gamma, \xi$ avoids $0$, $\check F$ is an immersion. Moreover,
    \begin{align*}
      \left\langle \frac{\partial \check F}{\partial x},  J\frac{\partial \check F}{\partial y}\right\rangle  = -\im\left(\overline{\gamma'(x)}\overline{\xi(y)}\xi'(y)\gamma(x) - \overline{\gamma'(x)}\gamma(x)\xi'(y)\overline{\xi(y)}\right) = 0,
    \end{align*}
    so $\check F$ is Lagrangian. The Lagrangian angle of $\check F$ is given by
    \begin{align*}
        \check \theta(x, y) = \arg\det\left(D\check F\right) &= \arg\left(\gamma'(x)\xi'(y)\left(|\gamma(x)|^{2}+|\xi(y)|^{2}\right)\right) \\
        &= \arg\left(\gamma'(x)\right)+\arg\left(\xi'(y)\right).
    \end{align*}
    Finally, we calculate the primitive of the pullback of the  Liouville form $\lambda$. From the definition of $\check F$, along with the fact that $\beta_\gamma'(x) = \im(\overline{\gamma(x)} \gamma'(x))$, we have 
    \begin{align*}
        \check F^*(\overline z_1 dz_1 - z_1 d\overline{z_1}) \, &= \, 2i \im(\overline{\gamma(x)}\gamma'(x))|\xi(y)|^2\,dx \, + \, 2i\im(\overline{\xi(y)}\xi'(y))|\gamma(x)|^2\,dy \\
        &= \, 2i \beta_\gamma'(x)|\xi(y)|^2 dx \, + \, 2i\beta_\xi'(y)|\gamma(x)|^2 dy, \\
        \check F^*(\overline{z_2}dz_2 - z_2 d\overline{z_2}) \, &= \, -2i\left(\beta_\xi(y) - \beta_\gamma(x)\right)\left(\xi(y)\xi'(y)dy - \gamma(x)\gamma'(x)dx\right) \, \\
        &\quad\qquad+ \, i\left(|\xi(y)|^2 - |\gamma(x)|^2\right)\left(\beta'_\xi(y)dy - \beta'_\gamma(x)dx\right),
    \end{align*}
    so that
    \begin{align*}
        \check F^*(\lambda) \, &= \, -\left(\beta_\xi(y) - \beta_\gamma(x)\right)\left(\xi(y)\xi'(y)dy \, - \, \gamma(x)\gamma'(x)dx\right) \:\\
        &\qquad\quad + \, \frac{1}{2} \left(|\xi(y)|^2 + |\gamma(x)|^2\right)\left(\beta'_\xi(y) dy \, + \, \beta'_\gamma(x)dx \right) \\ 
        &= \, -\frac{1}{2} \, d\left( \left(|\gamma(x)|^2 - |\xi(y)|^2\right)\left(\beta_\gamma(x) - \beta_\xi(y)\right) \right) \, + \, d\beta_\gamma(x) |\gamma(x)|^2 \, + \, d\beta_\xi(y) |\xi(y)|^2 \\
        &= \, d\left(
        -\frac{1}{2}\left(|\gamma(x)|^2 - |\xi(y)|^2\right)\left(\beta_\gamma(x) - \beta_\xi(y)\right) \, + \, \widetilde \beta_\gamma(x) \, +\,  \widetilde \beta_\xi(y)
        \right).
    \end{align*}
\end{proof}

We now find conditions on the curves $\gamma$, $\xi$ so that $\check F$ solves equations \eqref{eq-cohom2translator} and \eqref{eq-cohom2SL}. 

\begin{proposition}\label{prop-translatoransatz}
    Let $m\geq 2$, $I_{1}, I_{2}\subseteq\mathbb{R}$ be open intervals, $(a, b)\in\mathbb{R}^{2}$ be constants, and let $\gamma: I_{1}\to\mathbb{C}$,  $\xi:I_{2}\to\mathbb{C}$ be smooth curves satisfying
    \begin{align}\label{eq-curveeqs}
        \vec{\kappa}_{\gamma} - (m-2)\frac{\gamma^{\perp}}{|\gamma|^{2}} = a\gamma^{\perp},\quad \vec{\kappa}_{\xi} - (m-2)\frac{\xi^{\perp}}{|\xi|^{2}} = b\xi^{\perp}.
    \end{align}
    Then the Lagrangian immersion $\check F:I_{1}\times I_{2}\to\mathbb{C}^{2}$ given in Proposition \ref{prop-lagrangianansatz}, equation \eqref{eq-ansatz} has Lagrangian angle
    \begin{align}\label{eq-ansatzlagangleformula}
        \check \theta(x, y) = a_{1}\beta_{\gamma}(x) + b\beta_{\xi}(y) - (m-2)\arg(\gamma(x)\xi(y)) +c
    \end{align}
    for some $c\in\mathbb{R}$.
    In particular:
    \begin{itemize}
        \item If $a = b = 0$, then the immersion $\check F$ satisfies \eqref{eq-cohom2SL}.
        \item If $a = -b = 1$, then the immersion $\check F$ satisfies \eqref{eq-cohom2translator}.
    \end{itemize}
\end{proposition}

To prove Proposition \ref{prop-translatoransatz}, we first prove the following Lemma.

\begin{lemma}\label{lem: curve computations}
    Let $I\subseteq\mathbb{R}$ and let $\gamma:I\to\mathbb{C}$ be a smooth curve. Then
    \begin{equation}\label{eq: diff of arg of curve position}
        \frac{d}{ds}\arg(\gamma(s)) = \frac{d\beta_{\gamma}(s)}{ds}|\gamma(s)|^{-2}.
    \end{equation}
    Moreover, if $\gamma$ satisfies
    \begin{equation}
        \vec{\kappa}_{\gamma} - (m-2)\frac{\gamma(s)^{\perp}}{|\gamma(s)|^{2}} = a\gamma(s)^{\perp}
    \end{equation}
    for some $a\in\mathbb{R}$, then
    \begin{equation}\label{eq: diff of arg of curve tangent}
        \frac{d}{ds}\arg(\gamma'(s)) = -\left(\frac{m-2}{|\gamma(s)|^{2}}+a\right)\frac{d \beta_{\gamma}(s)}{ds}. 
    \end{equation}
\end{lemma}

\begin{proof}
    Let $\theta := \arg\gamma$. Write $\gamma = (\gamma\overline{\gamma})^{\frac{1}{2}}e^{i\theta}$, then we find
\begin{align*}
    \log\gamma = \frac{1}{2}\log\gamma + \frac{1}{2}\log\overline{\gamma} + i\theta.
\end{align*}
Thus,
\begin{align*}
    \frac{d}{ds}\theta = \frac{1}{2i}\left(\frac{\gamma'}{\gamma} - \frac{\overline{\gamma}'}{\overline{\gamma}}\right) = \frac{1}{|\gamma|^{2}}\frac{1}{2i}(\overline{\gamma}\gamma' - \gamma\overline{\gamma}')= -\frac{\im(\overline{\gamma}'(s)\gamma(s))}{|\gamma(s)|^{2}} = \frac{d\beta_{\gamma}(s)}{ds}|\gamma(s)|^{-2}.
\end{align*}
This proves (\ref{eq: diff of arg of curve position}). Next, we prove (\ref{eq: diff of arg of curve tangent}). Parametrising $\gamma$ by its arc-length gives
\begin{align*}
    \frac{d}{ds}\gamma'(s) \, = \,  \vec{\kappa}_{\gamma} = \left(\frac{m-2}{|\gamma(s)|^{2}}+a\right)\langle \gamma(s), i\gamma'(s)\rangle i\gamma'(s) \, = \,  -\left(\frac{m-2}{|\gamma(s)|^{2}}+a\right)\frac{d \theta}{ds} |\gamma(s)|^2 i\gamma'(s).
\end{align*}
The result follows by noting that $\frac{d}{ds}\arg(\gamma'(s)) \, = \, \langle \gamma'', \, i\gamma'\rangle$.
\end{proof}

\begin{proof}[Proof of Proposition \ref{prop-translatoransatz}]
From Lemma~\ref{lem: curve computations}, we have
\begin{align*}
    \frac{d}{dx}\arg\left(\gamma'(x)\right) &= -\left(\frac{m-2}{|\gamma(x)|^{2}} + a\right)\frac{d\beta_{\gamma}(x)}{dx}\\
    &= -(m-2)\frac{d}{dx}\arg\left(\gamma(x)\right) + a\frac{d}{dx}\beta_{\gamma}(x).
\end{align*}
Therefore,
\begin{equation}\label{eq: scalar eq for gamma}
    \arg\left(\gamma'(x)\right) + (m-2)\arg\left(\gamma(x)\right) = a\beta_{\gamma}(x) + c_{1}
\end{equation}
for some $c_{1}\in\mathbb{R}$. Similarly, we have
\begin{equation}\label{eq: scalar eq for xi}
        \arg\left(\xi'(y)\right) + (m-2)\arg\left(\xi(y)\right) = b\beta_{\xi}(y) + c_{2}
\end{equation}
for some $c_{2}\in\mathbb{R}$. Combining these with the expression for the Lagrangian angle in Proposition \ref{prop-lagrangianansatz} yields
\begin{align*}
    \check\theta(x, y) &= \arg\left(\gamma'(s)\right) + \arg\left(\xi'(y)\right)\\
    &= -(m-2)\arg\left(\gamma(x)\xi(y)\right) + a\beta_{\gamma}(x) + b\beta_{\xi}(y) + c
\end{align*}
for some $c = c_{1}+c_{2}\in\mathbb{R}$, and equation \eqref{eq-ansatzlagangleformula} is proved. Now, consider the case $a = -b \in\mathbb{R}$.
    By (\ref{eq: scalar eq for gamma}), (\ref{eq: scalar eq for xi}), we have
    \begin{align}
        a(\beta_{\gamma} - \beta_{\xi}) + c_{1}+c_{2} \, &= \,  \arg(\gamma'(x)\xi'(y)) + (m-2)\arg(\gamma(x)\xi(y)) \nonumber\\
        \implies a\im(z_{2})(x, y)  \,&=\, a(\beta_{\xi}(y) - \beta_{\gamma}(x))\nonumber\\
        &=\, -\left[\arg(\gamma'(x)\xi'(y)) + (m-2)\arg(\gamma(x)\xi(y)) - c_1 - c_2 \right]. \nonumber
    \end{align}
    Thus, 
    \begin{align}
        \check\theta(x, y) + (m-2)\arg(z_1)(x, y) + a\im(z_2)(x, y) = c_1 + c_2. \nonumber
    \end{align}
    The result now follows by choosing $a=0$ and $a=1$.

\end{proof}

Combining the results of this chapter with Corollary \ref{cor-cohom2translators}, we arrive at our main theorem.

\begin{theorem}\label{thm-slagsandtranslators}
    Assume we are in the setting of Assumption \ref{as1}, and let $\Phi:P \to \mathbb{C}$ be the isomorphism from Proposition \ref{prop-lagangle}. Let $I_1$ $I_2 \subset \mathbb{R}$ be open intervals, and $\gamma:I_1 \to \mathbb{C}$, $\xi:I_2 \to \mathbb{C}$ be smooth curves satisfying \eqref{eq-curveeqs} for $a, b \in \mathbb{R}$. Let $\check F$ be the map from Proposition \ref{prop-lagrangianansatz}.

    Then the image $L$ of the map 
    \[ F:G \times I_1 \times I_2 \to \mathbb{C}^{m-1} \times \mathbb{C}, \quad  F(g, x, y) = g \cdot (\Phi\oplus \text{Id})^{-1} \circ \check F(x,y) \]
    is an exact cohomogeneity-two $G$-invariant Lagrangian submanifold, immersed away from $\{0\} \times \mathbb{C}$. The restriction of the Liouville form $\lambda$ is given by:
    \begin{equation}
        F^*\lambda = d\left( -\frac{1}{2}(|\gamma(x)|^2 - |\xi(y)|^2)(\beta_\gamma(x) - \beta_\xi(y)) + \widetilde \beta_\gamma(x) + \widetilde \beta_\xi(y) \right).
    \end{equation} 
    Furthermore:
    \begin{itemize}
        \item If $a = b = 0$ then $L$ is a special Lagrangian, i.e. $\theta(F) = 0$.
        \item If $a = -b = 1$ then $L$ is a Lagrangian translator, i.e. $\theta(F) = -\langle F, Je_m\rangle$.
    \end{itemize}
\end{theorem}
\begin{proof}
    The only statement still requiring justification regards the restriction of the Liouville form. This follows from equation \eqref{eq-restrictionofliouville}, along with the following claim, denoting by $\Psi$ the map $\Psi := (\Phi \oplus \text{Id})^{-1}$ and by $\pi$ the projection $\pi: G \times I_1 \times I_2 \to I_1 \times I_2$:
    \[ F^* \lambda \, = \, \pi^* \check F^* \Psi^* \lambda.\]
    To prove this claim, choose $(g,x,y) \in G \times I_1 \times I_2$, and denote $w := g \cdot \Psi\circ \check F(x,y)$. We evaluate each side of the equation in two cases:
    \begin{itemize}
        \item[a)] Assume $X \in T_g G$, so that $F_*(X) \in T_w \mathcal{O}_w$. Then by equation \eqref{eq-ambientspacedecomposition}, $F_*(X), JF_*(X) \perp P_w$, which implies that $\lambda(F_* X) = -\langle w, J F_* X\rangle \, = \, 0$. This matches the right-hand side as $\pi_* X = 0$.
        \item[b)] Assume $X \in T_{(x,y)}I_1 \times I_2$. Then by $G$-invariance of $\lambda$, 
        \[\lambda_w(F_* X) = \lambda_{\Psi \circ \check F(x,y)}((\Psi\circ \check F)_*X) \, = \,  (\pi^*\check F^* \Psi^* \lambda)(X). \]
    \end{itemize}
\end{proof}

\subsection{Classification of Admissible Curves}\label{sec-classification1}

We consider solutions to \eqref{eq-curveeqs}. For $a>0$, (\ref{eq-curveeqs}) implies that $\gamma$ is the profile curve of a cohomogeneity-one Lagrangian expander in $\mathbb{C}^{m-1}$, and that $\xi$ is the profile curve of a cohomogeneity-one Lagrangian shrinker in $\mathbb{C}^{m-1}$. If $a = 0$, $\gamma$ and $\xi$ are profile curves of cohomogeneity-one special Lagrangians. Such curves have been fully classified by Anciaux, Castro and Romon (see \cite{Anciaux2006, AnciauxRomon2009, AnciauxCastroRomon2006}). We summarise the classification here.

\begin{theorem}[Classification of Solutions to \eqref{eq-curveeqs}]\label{thm-curveeqssolutions}
    Up to rotation by $e^{i\phi}$, the only complete, properly immersed solutions to the equation
    \begin{equation}\label{eq-curveeq}
        \kappa_\gamma - (m-2)\frac{\gamma^\perp}{|\gamma^2|} = a\gamma^\perp
    \end{equation} 
    are:
    \begin{align*}
        &\text{Line, }a \in \mathbb{R}: 
         &&\tilde{c} (s) := s, \text{ for } s \in \mathbb{R}\\
        &\text{Lawlor neck profile, }a=0: &&\tilde l(s) := \frac{1}{\sqrt[m-1]{\cos((m-1)s)}\quad}\,e^{is}, \text{ for } s \in (-\tfrac{\pi}{2(m-1)}, \tfrac{\pi}{2(m-1)})\\
        &\text{(along with its scalings by $\mu \in \mathbb{R}^+$})\\
        &\text{Anciaux shrinker profile, } a<0: &&\tilde \sigma_{b,p,q}(s), \text{ for } s \in \mathbb{R}\\
        &\text{Anciaux expander profile, } a>0: &&\tilde \epsilon_{a,\alpha}(s), \text{ for } s \in (-\tfrac{\alpha}{2}, \tfrac{\alpha}{2})        
    \end{align*}
    (as well as unions thereof), where:
    \begin{itemize}
        \item $\tilde \sigma_{b,p,q}$ is the profile curve of the Anciaux shrinker at scale $b = -a$ with winding number $p$ and number of points of maximum curvature $q$ (see Figure \ref{fig-shrinker}).
        \item $\tilde \epsilon_{a,\alpha}$ is the profile curve of the Anciaux expander at scale $a$, centred on the positive real axis, spanning an angle $\alpha \in (0, \frac{\pi}{m-1})$ (see Figure \ref{fig-expander}).
    \end{itemize}
\end{theorem}

\begin{figure}[t]
  \centering
  \begin{subfigure}[t]{0.3\textwidth}
    \centering
    \includegraphics[width=\linewidth]{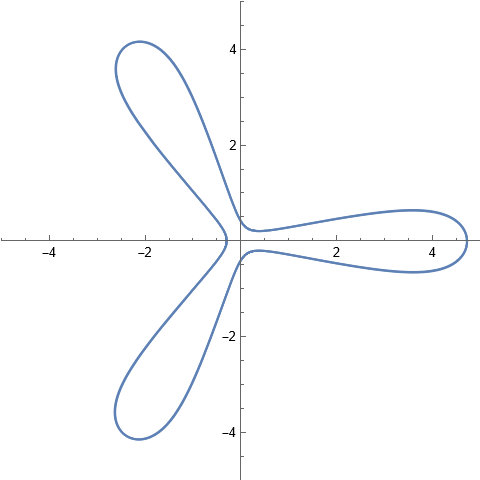}
    \caption{An Anciaux shrinker profile curve $\tilde \sigma_{a, p, q}$, with $p=1$, $q=3$.}
    \label{fig-shrinker}
  \end{subfigure}
  \hfill
  \begin{subfigure}[t]{0.3\textwidth}
    \centering
    \includegraphics[width=0.8\linewidth]{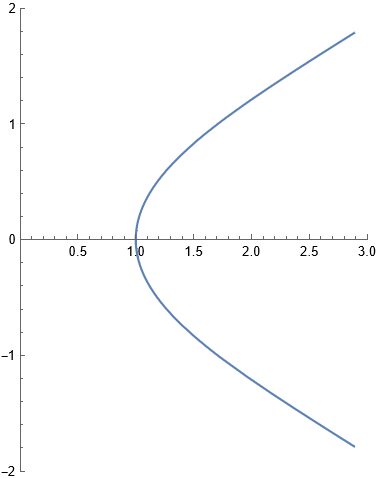}
    \caption{An Anciaux expander profile curve $\tilde \epsilon_{a,\alpha}$, with $\alpha \approx 1.1$.}
    \label{fig-expander}
  \end{subfigure}
  \hfill
  \begin{subfigure}[t]{0.3\textwidth}
    \centering
    \includegraphics[width=0.8\linewidth]{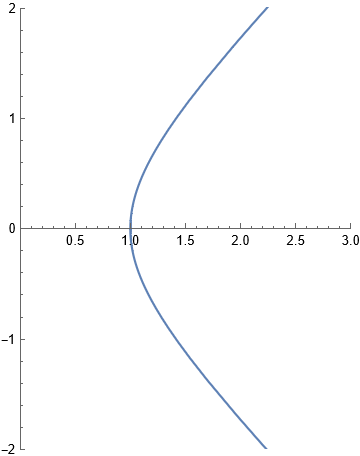}
    \caption{The Lawlor neck profile curve $\tilde l$.}
    \label{fig-lawlor}
  \end{subfigure}
  \caption{Three particular solutions to \eqref{eq-curveeq}, all in the case $m=3$.}
  \label{fig-curves}
\end{figure}

\noindent See \cite{Anciaux2006} for the derivation of the Anciaux shrinkers and expanders. Note that in the notation of Anciaux, $a = -b = -\lambda$ and $m-1 = n$.

\begin{remark}\label{rem-betaforthecurves}
    We briefly discuss the properties of $\beta_\gamma(s)$ for each of these curves. For the line $\tilde c$, $\beta$ is constant, so we may take it to be zero. Each of the other curves $\tilde l$, $\tilde \sigma_{b,p,q}$, $\tilde \epsilon_{a, \alpha}$ is a polar graph $\gamma(s) = r(s)e^{is}$ (e.g. the curves of Figure \ref{fig-curves}). It follows from Remark \ref{rem-betapositivity} that $\beta = \int_{s_0}^s |r(t)|^2 \,dt$, and is strictly increasing. 
    
    In the case of $\tilde l$ and $\tilde \epsilon_{a, \alpha}$ we choose $s_0 = 0$ so that $\beta(s) = -\beta(-s)$. For $\tilde \sigma_{b,p,q}$, it is clear that $\beta$ is unbounded in the positive and negative directions. For $\epsilon_{a,\alpha}$, one can integrate \eqref{eq-curveeqs} to find that the range of $\beta$ is $\frac{1}{a}(\pi - (m-1)\alpha)$.
\end{remark}

\subsection{Cohomogeneity-Two Special Lagrangians}\label{sec-classification2}

We now consider what properties the Lagrangian immersion of Proposition \ref{prop-translatoransatz} and corresponding cohomogeneity-two Lagrangian immersion of Theorem \ref{thm-slagsandtranslators} have with these choices of curves $\gamma, \xi$. For $a=b=0$, the options for the pair $(\gamma, \xi)$ are (without loss of generality, where $\phi \in \mathbb{R},\, \mu \in \mathbb{R}^+$):
\begin{itemize}
    \item[(a)] $(\gamma = e^{i\phi}\tilde c,\, \xi = \tilde c)$, i.e. both curves are lines through $0$. Then the image of the mapping
    \begin{equation}\label{eq-Fa}
        \check F_{(a)}(x,y) = \left(e^{i\phi}xy,\,\frac{1}{2}(y^2 - x^2)\right)
    \end{equation} 
    is simply a flat copy of $\mathbb{R}^2$ in 
    $\mathbb{C}^2$.
    \item[(b)] $(\gamma = \mu \tilde l, \, \xi = e^{i\phi}\tilde c)$, i.e. $\gamma$ is a Lawlor neck profile and $\xi$ is a line. Then:
    \begin{equation}\label{eq-Fb}
    \check F_{(b)}(x,y) = \left(e^{i\phi}\mu \, y\,  \tilde l (x),\,\frac{1}{2}(y^2-\mu^2|\tilde l (x)|^2) \, - \, i\mu^2 \beta_{\tilde l}(x) \right).
    \end{equation}
    \item[(c)] $(\gamma = e^{i\phi}\mu_1 \tilde l, \, \xi = \mu_2 \tilde l)$, i.e. $\gamma$ and $\xi$ are both Lawlor neck profiles. Then:
    \begin{equation}\label{eq-Fc}
    \check F_{(c)}(x,y) = \left(e^{i\phi}\mu_1\mu_2 \, \tilde l(x)\,  \tilde l (y),\,\frac{1}{2}(\mu_2^2|\tilde l (y)|^2 - \mu_1^2|\tilde l(x)|^2) \, + \, i(\mu_2^2 \beta_{\tilde l}(y) - \mu_1^2\beta_{\tilde l}(x) )\right).
    \end{equation}
    
\end{itemize}
The latter two can be seen to be exact Lagrangian embeddings of $\mathbb{R}^2$ into $\mathbb{C}^2$, by checking that the map $\check F$ is injective. By Theorem \ref{thm-slagsandtranslators}, these are `profile surfaces' corresponding to special Lagrangians in $\mathbb{C}^m$. Whether these special Lagrangians are immersed/embedded or not depends on the group $G$. In total, we have the following:

\begin{proposition}\label{prop-slclassification}
    In the setting of Assumption \ref{as1}, define the Lagrangians $L_{(b)}$, $L_{(c)}$ corresponding to the maps $\check F_{(b)}$, $\check F_{(c)}$ in \eqref{eq-Fb}, \eqref{eq-Fc} as the image of the map $F$ of Theorem \ref{thm-slagsandtranslators}.
    
    Then $L_{(b)}$, $L_{(c)}$ are exact special Lagrangian submanifolds. Furthermore: 
    \begin{itemize}
        \item $L_{(b)}$ is embedded if $G = \SO(m-1)$, and otherwise embedded away from a 1-dimensional singular set corresponding to $\{x=0\}$,
        \item $L_{(c)}$ is immersed for any $G$, and an embedding provided that the discrete group $C_k$ has order $k\leq m-1$.
    \end{itemize}
\end{proposition}

\begin{proof}
    Given $\check F:I_1\times I_2 \to\mathbb{C}^2$, we define the map $F:G \times I_1 \times I_2 \to \mathbb{C}^m$ as in Theorem \ref{thm-slagsandtranslators}. 
    
    In the case of $(b)$, since $(\Phi \oplus \text{Id})^{-1}(\check L)$ intersects the singular orbit $\{0\}\times \mathbb{C} \subset \mu_{\mathbb{C}^m}^{-1}(0)$, the only way the image $L$ of $F$ can be nonsingular is if $G = \SO(m-1)$, in which case $L$ has the topology of $\mathbb{R}^m$ near $x=0, y=0$. In this case, embeddedness of $L$ corresponds to embeddedness of $C_2 \cdot \check L \subset \mathbb{C}^2$. Since $\check L_{(b)}$ is $C_2$-symmetric by the domain mapping $x \mapsto -x$, this is clear.

    For $(c)$, $\check L$ does not intersect the singular orbit, so we may consider a general group $G$ satisfying $|C_k| \leq m-1$. Then by Proposition \ref{prop-correspondence}, $L$ is immersed, and embedded if and only if $C_k \cdot \check L$ is embedded. To show this, assume that $e^{i\varphi}\cdot \check F_{(c)}(x_1, y_1) \, = \, \check F_{(c)}(x_2, y_2)$, for $e^{i\varphi} \in C_k$. Using polar coordinates $\tilde l(x) = r(x)e^{ix}$, this gives the following four equations:
    \begin{itemize}
        \item[1.] $\varphi + x_1 + y_1 \, \equiv \, x_2 + y_2 \quad (\text{mod }2\pi)$,
        \item[2.] $\mu_1\mu_2r(x_1)r(x_2) \, = \, \mu_1 \mu_2 r(x_2)r(y_2)$,
        \item[3.] $\mu_2^2r(y_1)^2 - \mu_1^2 r(x_1)^2 \, = \, \mu_2^2r(y_2)^2 - \mu_1^2 r(x_2)^2$,
        \item[4.] $\mu_2^2 \beta_{\tilde l}(y_1) - \mu_1^2 \beta_{\tilde l}(x_1) \, = \, \mu_2^2 \beta_{\tilde l}(y_2) - \mu_1^2 \beta_{\tilde l}(x_2)$.
    \end{itemize}
    Combining equations 2 and 3 gives the single complex equation
    \begin{align}
         (\mu_2 r(y_1) + i\mu_1 r(x_1))^2 \, &= \, (\mu_2 r(y_2) + i\mu_1 r(x_2))^2 \notag \\
         \implies r(y_1) \, = r(y_2), \,\,r(x_1) \, &= \, r(x_2). \label{eq-sameradius}
    \end{align}
    \indent Assume first that $x_1 \neq x_2$, without loss of generality we may assume $x_1 > x_2$. By \eqref{eq-sameradius} we see that $x_2 = -x_1$, and by Remark \ref{rem-betaforthecurves} we have $\beta_{\tilde l}(x_2) = -\beta_{\tilde l}(x_1)$. Equation 4 now gives that $\mu_2^2(\beta_{\tilde l}(y_1) - \beta_{\tilde l}(y_2)) \, = \, 2\mu_1^2 \beta_{\tilde l}(x_1) \, > 
    \, 0$, so we must have $y_1 > y_2$, indeed by \eqref{eq-sameradius} $y_2 = -y_1$. Now we have $ 0 < 2x_1 + 2y_1 < \tfrac{2\pi}{m-1}$, which contradicts equation 1 since $\varphi + 2x_1 + 2y_1$ cannot be an integer for $e^{i\varphi} \in C_k$, $k \leq m-1$. 
    
    Therefore, we must have $x_1 = x_2$, which by equations 1 and 4 imply $y_1 = y_2$ and $e^{i\varphi} = 1$.
\end{proof}

\begin{remark} For example, $\check F_{(c)}$ corresponds to a special Lagrangian embedding when $G = SO(m-1)$ and $T^{m-2}$, since in those cases $C_k = C_2$ and $C_{m-1}$ respectively (see Examples \ref{ex-so(m-1)} and \ref{ex-t(m-2)}). Indeed, since in general we have $k|2(m-1)$ by Proposition \ref{prop-levelsetgeometry} c), the assumption $k \leq m-1$ only fails when $k = 2(m-1)$. The authors are unaware of any group actions with this property.
\end{remark}

\subsection{Cohomogeneity-Two Lagrangian Translators}\label{sec-classification3}

For $a=-b = 1$, the options for the pair $(\gamma, \xi)$ are (without loss of generality, and suppressing the subscripts $a, p, q, \alpha$ for the Anciaux curves):

\begin{itemize}
    \item[d)] $(\gamma = e^{i\phi}\tilde c, \, \xi = \tilde \sigma)$, i.e. $\gamma$ is a line and $\xi$ is an Anciaux shrinker profile. Then:
    \begin{equation}\label{eq-Fd}
        \check F_{(d)}(x,y) = \left(e^{i\phi} \,x \, \tilde \sigma(y), \, \frac{1}{2}(|\tilde \sigma(y)|^2 - x^2) \, + \, i\beta_{\tilde \sigma}(y)\right).  
    \end{equation}
    \item[e)] $(\gamma =  \tilde\epsilon, \, \xi = e^{i\phi}\tilde c)$, i.e. $\gamma$ is an Anciaux expander profile and $\xi$ is a line. Then:
    \begin{equation}\label{eq-Fe}
    \check F_{(e)}(x,y) = \left(e^{i\phi} \,y \, \tilde \epsilon(x), \, \frac{1}{2}(y^2 - |\tilde \epsilon(x)|^2 ) \, - \, i\beta_{\tilde \epsilon}(x)\right).
    \end{equation}
    \item[f)] $(\gamma = e^{i\phi} \tilde \epsilon(x), \, \xi = \tilde \sigma(y))$, i.e. $\gamma$ is an Anciaux expander profile and $\xi$ is an Anciaux shrinker profile. Then:
    \begin{equation}\label{eq-Ff}
    \check F_{(f)}(x,y) = \left(e^{i\phi}  \, \, \tilde \epsilon(x) \tilde \sigma(y), \, \frac{1}{2}(|\tilde \sigma(y)|^2 - |\tilde \epsilon(x)|^2 ) \, + \, i(\beta_{\tilde \sigma}(y) - \beta_{\tilde \epsilon}(x))\right).
    \end{equation}
\end{itemize}
Among these, d) and e) can be seen to be exact Lagrangian embeddings of $\mathbb{R}^2$ into $\mathbb{C}^2$ as in the $a=b=0$ case, and $f)$ can be seen to be an exact Lagrangian immersion. Note that this is true for d) despite $\tilde \sigma_{a,p,q}$ not necessarily being embedded itself for all values of $(p, q)$ - the Ansatz `unwinds' the shrinker into an embedded surface.

By Theorem \ref{thm-slagsandtranslators}, the immersions $\check F$ in d), e) and f) are `profile surfaces' corresponding to Lagrangian translators in $\mathbb{C}^m$. Note that the Lagrangian angle of the corresponding Lagrangian translators is negative the imaginary part of the second coordinate of $\check F$. Therefore, the translators corresponding to d) and f) are never almost-calibrated, whereas Remark \ref{rem-betaforthecurves} shows the translator corresponding to e) is almost-calibrated. Depending on the group $G$, the corresponding translators are embedded submanifolds of $\mathbb{C}^m$:

\begin{proposition}\label{prop-translatorclassification}
    In the setting of Assumption 1, define the Lagrangians $L_{(d)}$, $L_{(e)}$, $L_{(f)}$ corresponding to the maps $\check F_{(d)}$, $\check F_{(e)}$, $\check F_{(f)}$ in \eqref{eq-Fd}, \eqref{eq-Fe}, \eqref{eq-Ff} respectively as the image of the map $F$ of Theorem \ref{thm-slagsandtranslators}. 
    
    Then $L_{(d)}$, $L_{(e)}$, $L_{(f)}$ are exact zero-Maslov Lagrangian translating solitons, and $L_{(e)}$ is almost-calibrated. Furthermore: 
    \begin{itemize}
        \item $L_{(d)}, L_{(e)}$ are embeddings if $G = \SO(m-1)$, and otherwise embeddings away from a 1-dimensional singular set corresponding to $\{x=0\} , \{y=0\}$ respectively,
        \item $L_{(f)}$ is an immersion for any $G$.
    \end{itemize}
\end{proposition}

\begin{remark}
    The example $L_{(e)}$ with $G = \SO(m-1)$ is the symmetric Joyce--Lee--Tsui translator (see \cite{Joyce2010}). By considering a suitable sequence of expander profile curves $\epsilon_{a_i, \alpha_i}$ converging to a Lawlor neck profile curve $\tilde l$, the translator $L_{(e)}$ converges to the special Lagrangian $L_{(b)}$.
\end{remark}

\begin{proof}
    We sketch the argument for $\check F_{(d)}$ only, in the case where $G = \SO(m-1)$. Since $\check F$ has $C_2$-symmetry achieved by the mapping $x \mapsto -x$, to show embeddedness we must show that if $e^{i\varphi}\check F_{(d)}(x_1,y_1) = \check F_{(d)}(x_2, y_2)$, then either $\varphi \equiv 0 \,\,(\text{mod } 2\pi), x_1 = x_2, y_1 = y_2$, or $\varphi \equiv \pi \,\,(\text{mod } 2\pi), x_1 = -x_2, y_1 = y_2$.
    When the equality holds, we have the three equations
    \begin{align*}
        e^{i\varphi}x_1 \tilde \sigma(y_1) = x_2 \tilde \sigma(y_2), \quad |\tilde \sigma(y_1)|^2 - x_1^2 = |\tilde \sigma(y_2)|^2 - x_2^2, \quad \beta_{\tilde \sigma}(y_1) = \beta_{\tilde \sigma}(y_2).
    \end{align*}
    By Remark \ref{rem-betaforthecurves}, $\beta_{\tilde \sigma}$ is strictly increasing, and so we must have $y_1 = y_2$. The result follows from the other two equations.
\end{proof}

\section{Affine-Invariant Cohomogeneity-One Lagrangian Translators} \label{sec-moreexamples}

Finally, we describe a construction of explicit cohomogeneity-one examples of Lagrangian translators corresponding to an affine group action, using the level set method by Harvey--Lawson \cite{HL82}.

\subsection{$\widetilde{\U(1)}^{2}$-invariant Examples}

Consider the action of $G := \widetilde{\U(1)}^{2}\simeq\mathbb{R}^{2}$ on $\mathbb{C}^{3}$ by
\begin{align*}
	(e^{i\theta_{1}}, e^{i\theta_{2}})\cdot(z_{1}, z_{2}, z_{3}) = (e^{i\theta_{1}}z_{1}, e^{i\theta_{2}}z_{2}, z_{3} - i(\theta_{1} + \theta_{2})).
\end{align*}
By Proposition \ref{prop-affinegrouppreservingomegaf}, $\widetilde {\U(1)}^2 \leq G_f$, and so the action preserves $\Omega_f$. Define a map $\mu:\mathbb{C}^{3}\to\mathbb{R}^{3}$  by
\begin{align*}
	\mu(z_{1}, z_{2}, z_{3}) = \left(\re z_{3} - \frac{1}{2}|z_{1}|^{2}, \re z_{3} - \frac{1}{2}|z_{2}|^{2}, \im\left(z_{1}z_{2}e^{z_{3}}\right)\right).
\end{align*}
Let $\mu = (\mu_{1}, \mu_{2}, \mu_{3})$, then in terms of real coordinates of  $\mathbb{C}^{3}\simeq\mathbb{R}^{6}$,
\begin{align*}
	&\mu_{1}(z_{1}, z_{2}, z_{3}) := \frac{1}{2}(z_{3} + \overline{z_{3}} - |z_{1}|^{2}) = x_{3} - \frac{x_{1}^{2} + y_{1}^{2}}{2},\\
	&\mu_{2}(z_{1}, z_{2}, z_{3}) := \frac{1}{2}(z_{3} + \overline{z_{3}} - |z_{2}|^{2}) = x_{3} - \frac{x_{2}^{2}+y_{2}^{2}}{2},\\
	&\mu_{3}(z_{1}, z_{2}, z_{3}) = \frac{1}{2i}(z_{1}z_{2}e^{z_{3}} - \overline{z_{1}z_{2}}e^{\overline{z_{3}}}) = e^{x_{3}}[(x_{1}y_{2}+x_{2}y_{1})\cos y_{3}+(x_{1}x_{2} - y_{1}y_{2})\sin y_{3}].
\end{align*}
It is straightforward to check that $d\mu:\mathbb{C}^{3}\to\mathbb{R}^{3}$ is surjective if and only if $(z_1 , z_2)\neq (0, 0)$, and the Poisson brackets $\{\mu_{j}, \mu_{k}\}$ vanish for all $j, k = 1, 2, 3$. Hence, by \cite[Section III.2.C]{HL82}, the smooth part of $L_{a, b, c} := \mu^{-1}(a, b, c)$ is Lagrangian. Notice that $(\mu_{1}, \mu_{2})$ is the moment map of the $G
$ action, and $\mu_{3}$ can be viewed as a generalized moment map corresponding to $\im\Omega_{f}$. Now $L_{a, b, c}$ is the intersection of $3$ level sets $\mu_{1}^{-1}(a)\cap\mu_{2}^{-1}(b)\cap\mu_{3}^{-1}(c)$, so it is generically real $3$-dimensional. Let $z = (z_{1}, z_{2}, z_{3})\in L_{a, b, c}$, then $T_{z}^{\perp}L_{a, b, c}$ is spanned by $\{\nabla\mu_{1}(z), \nabla\mu_{2}(z), \nabla\mu_{3}(z)\}$. Since $L_{a, b, c}$ is Lagrangian, $T_{z}L_{a, b, c}$ is spanned by
\begin{align*}
	J\nabla\mu_{j} = J\left\{\sum_{k=1}^{3}\frac{\partial\mu_{j}}{\partial x_{k}}\frac{\partial}{\partial x_{k}} + \frac{\partial\mu_{j}}{\partial y_{k}}\frac{\partial}{\partial y_{k}}\right\} = \sum_{k=1}^{3}\frac{\partial\mu_{j}}{\partial x_{k}}\frac{\partial}{\partial y_{k}} - \frac{\partial\mu_{j}}{\partial y_{k}}\frac{\partial}{\partial x_{k}},\quad j = 1, 2, 3.
\end{align*}
Note that 
\begin{align*}
	2i\frac{\partial\mu_{j}}{\partial\overline{z}_{k}} = -\frac{\partial\mu_{j}}{\partial y_{k}}+ i\frac{\partial\mu_{j}}{\partial x_{k}},
\end{align*}
so the matrix $M_{jk} = 2i\frac{\partial\mu_{j}}{\partial\overline{z}_{k}}$ sends the standard basis $\{e_{1}, e_{2}, e_{3}\}$ to $\{J\nabla\mu_{1}(z), J\nabla\mu_{2}(z), J\nabla\mu_{3}(z)\}$. Therefore, the Lagrangian angle $\theta$ is given by the argument of
\begin{align*}
	\det M_{jk} = \frac{e^{\overline{z}_{3}}}{8i}\left(|z_{1}|^{2} + |z_{2}|^{2} + |z_{1}|^{2}|z_{2}|^{2}\right),
\end{align*}
which is $-\im z_{3} - \frac{\pi}{2}$. Hence, we have $\theta + \im z_{3} = -\frac{\pi}{2}$ is a constant, that is, (the smooth part of) $L_{a, b, c}$ is a Lagrangian translating soliton.

\subsection{A Lagrangian translator fibration}

Since $d\mu:\mathbb{C}^{3}\to\mathbb{R}^{3}$ is not surjective on $\{(0, 0, z_3)\:|\:z_3\in\mathbb{C}\}$, the fibers $L_{a, b, c} := \mu^{-1}(a, b, c)$ over $\Gamma := \{\mu(0, 0, z_3)\:|\:z_3\in\mathbb{C}\} = \{(a, a, 0)\:|\:a\in\mathbb{R}\}$ are singular. Indeed, for any $a\in\mathbb{R}$,
\begin{align*}
	L_{a, a, 0} &= \left\{\left(z_{1}, z_{2}, z_{3}\right)\::\:\re z_{3} - \frac{1}{2}|z_{1}|^{2} = \re z_{3} - \frac{1}{2}|z_{2}|^{2} = a,\; \im(z_{1}z_{2}e^{z_{3}}) = 0\right\}\\
	&= \left\{\left(z_{1}, z_{2}, z_{3}\right)\::\:|z_{1}|^{2} = |z_{2}|^{2} = r^{2}, \;\re z_{3} = a+\frac{r^{2}}{2},\; \im(z_{1}z_{2}e^{z_{3}}) = 0\right\}\\
	&=\left\{\left(re^{i\theta_{1}}, re^{i\theta_{2}}, a+\frac{r^{2}}{2} -i(\theta_{1} + \theta_{2}) \right)\::\:r\geq 0, \;(\theta_{1}, \theta_{2})\in\mathbb{R}^{2}\right\}
\end{align*}
has a $1$-dimensional singular set $\{(0, 0, a + it)\:|\:t\in\mathbb{R}\}$, and it is \emph{Hamiltonian stationary}, namely, its Lagrangian angle is harmonic. 

We now describe the geometry of a general fiber. Let $(a, b, c)\in\mathbb{R}^{3}$. Then
\begin{align}
    L_{a, b, c} &= \left\{\left(z_{1}, z_{2}, z_{3}\right)\::\:\re z_{3} - \frac{1}{2}|z_{1}|^{2} = a, \re z_{3} - \frac{1}{2}|z_{2}|^{2} = b,\; \im(z_{1}z_{2}e^{z_{3}}) = c\right\}\nonumber\\
    &=\left\{\left(z_{1}, z_{2}, \frac{a+b}{2} + \frac{|z_1|^{2}+|z_2|^{2}}{4} - i\theta(z_1, z_2)\right)\::\:|z_1|^{2} - |z_2|^{2} = 2(b-a)\right\},
\end{align}
where the Lagrangian angle is given by
\begin{align}\label{eq: imaginary part of Labc}
    \theta(z_1, z_2) = \arg(z_1 z_2)-\arcsin\left[\frac{c\cdot {\exp\left(-\frac{a+b}{2}-\frac{|z_1|^{2}+|z_2|^{2}}{4}\right)}}{|z_1 z_2|}\right].
\end{align}
One can view $L_{a, b, c}$ as a graph over the  submanifold
\begin{align}
    M_{2(b-a)} := \{(z_1, z_2)\in\mathbb{C}^{2}\::\:|z_1|^{2} - |z_2|^{2} = 2(b-a)\}.
\end{align}
Note that $M_{2(b-a)}$ is diffeomorphic to $\SU(1, 1)\simeq \text{SL}_{2}(\mathbb{R})$ if $a\neq b$, and is a $\mathbb{T}^{2}$-cone $M_{0}$ if $a = b$. Now $\text{SL}_{2}(\mathbb{R})$ acts on the upper half plane $\mathbb{H}^{2}$ transitively by M\"obius transform, with stabilizer $\SO(2)\simeq\mathbb{S}^{1}$. Hence, for $a\neq b$, $M_{2(b-a)}$ is a circle bundle over $\mathbb{H}^{2}$, which is diffeomorphic to $\mathbb{R}^{2}\times\mathbb{S}^{1}$. The translator $L_{a, b, c}$ is diffeomorphic to $\mathbb{R}^{3}$ if $a\neq b$, and it can be visualized as a `helical staircase', with each `step' corresponding to the graph of the function $h(z_1, z_2) = \frac{|z_2|^{2}+|z_2|^{2}}{4}$ over $M_{2(b-a)}$. To make $L_{a, b, c}$ complete, one needs to lift $M_{2(b-a)}$ to its universal cover, which lift its circle fibers to $\mathbb{R}$, and one also needs to lift the argument function in (\ref{eq: imaginary part of Labc}) to be multivalued.

\bibliographystyle{amsplain}
\bibliography{reference}

\end{document}